\documentclass{compositio}

\usepackage{amsmath,amsfonts,yhmath,hyperref}
\usepackage{amsthm}
\usepackage{amssymb}
\usepackage[latin9]{inputenc}
\usepackage[nottoc,notlof,notlot]{tocbibind}
\usepackage{mathtools}
\usepackage{enumitem}
\usepackage{tikz-cd}
\usepackage{multirow}
\usepackage{mathrsfs}

\usepackage{pgf,tikz}

\usetikzlibrary{arrows}
\usetikzlibrary[patterns]

\definecolor{qqqqff}{rgb}{0.,0.,1.}
\definecolor{cqcqcq}{rgb}{0.7529411764705882,0.7529411764705882,0.7529411764705882}
\definecolor{ttqqqq}{rgb}{0.2,0.,0.}
\definecolor{qqqqff}{rgb}{0.,0.,1.}
\definecolor{xdxdff}{rgb}{0.49019607843137253,0.49019607843137253,1.}
\definecolor{zzttqq}{rgb}{0.6,0.2,0.}
\definecolor{cqcqcq}{rgb}{0.7529411764705882,0.7529411764705882,0.7529411764705882}

\definecolor{yqyqyq}{rgb}{0.5019607843137255,0.5019607843137255,0.5019607843137255}
\definecolor{uuuuuu}{rgb}{0.26666666666666666,0.26666666666666666,0.26666666666666666}
\definecolor{xdxdff}{rgb}{0.49019607843137253,0.49019607843137253,1.}
\definecolor{qqqqff}{rgb}{0.,0.,1.}

\newcommand{\trop}{\mathrm{trop}}

\newcommand{\NN}{\mathbb{N}}
\newcommand{\ZZ}{\mathbb{Z}}
\newcommand{\RR}{\mathbb{R}}
\newcommand{\CC}{\mathbb{C}}

\newcommand{\TT}{\mathbb{T}}

\renewcommand{\P}{\mathcal{P}}
\renewcommand{\L}{\mathcal{L}}
\newcommand{\N}{\mathcal{N}}
\newcommand{\M}{\mathcal{M}}

\newcommand{\T}{\mathcal{T}}

\newcommand{\Q}{\mathcal{Q}}
\renewcommand{\S}{\mathcal{S}}

\newcommand{\F}{\mathcal{F}}

\newcommand{\dd}{ \mathrm{d} }

\newcommand{\im}{\mathfrak{Im}}

\newtheorem{theo}{Theorem}[section]
\newtheorem*{theom}{Theorem}
\newtheorem{prop}[theo]{Proposition}
\newtheorem{coro}[theo]{Corollary}
\newtheorem{lem}[theo]{Lemma}

\theoremstyle{definition}
\newtheorem{defi}[theo]{Definition}

\theoremstyle{remark}
\newtheorem{remark}[theo]{Remark}
\newenvironment{rem}[1]{
    \begin{remark}#1}{
    \xqed{\blacklozenge}\end{remark}
}

\theoremstyle{remark}
\newtheorem{example}[theo]{Example}
\newenvironment{expl}[1]{
    \begin{example}#1}{
    \xqed{\lozenge}\end{example}
}

\newcommand{\xqed}[1]{
    \leavevmode\unskip\penalty9999 \hbox{}\nobreak\hfill
    \quad\hbox{\ensuremath{#1}}}

\keywords{Enumerative geometry, tropical refined invariants, abelian surfaces\\
\texttt{Thomas Blomme, Univerit\'e de Gen\`eve} \\
Email: \texttt{thomas.blomme@unige.ch} \\
 \textit{Data Statement:} I do not have any data to point.}
 
\begin{document}
 
 
\title{Tropical curves in abelian surfaces I:\\ enumeration of curves passing through points}
\author{Thomas Blomme}

\begin{abstract}
This paper is the first part in a series of three papers devoted to the study of enumerative invariants of abelian surfaces through the tropical approach. In this paper, we consider the enumeration of genus $g$ curves of fixed degree passing through $g$ points. We compute the multiplicity provided by a correspondence theorem due to T. Nishinou and show that it is possible to refine this multiplicity in the style of Block-G\"ottsche to get tropical refined invariants.
\end{abstract}

\maketitle

\tableofcontents

\section{Introduction}

\subsection{Overview}

This paper is the first of a series of three papers whose purpose is to study the enumerative invariants of abelian surfaces. The first paper is dedicated to presenting the setting of tropical abelian surfaces and tropical curves, and studying the enumeration of genus $g$ curves passing through $g$ points. The second paper focuses on the enumeration of genus $g$ curves belonging to a fixed linear system that pass through $g-2$ points. The third paper provides a pearl diagram algorithm that enables a concrete computation of the invariants introduced in the first two papers.

\paragraph{Abelian surfaces and tropical tori.} We consider tropical tori. These are quotients of a real vector space containing some lattice, by some other lattice. Namely, we choose the real vector space to be $N_\RR=N\otimes\RR$ for some lattice $N$, and the other lattice is denoted by $\Lambda\subset N_\RR$. The inclusion $\Lambda\hookrightarrow N_\RR$ is denoted by $S$. The case of surfaces corresponds to the case where both lattices are of rank $2$. This construction is analog to the construction of complex tori as a quotient of the algebraic torus $(\CC^*)^n$ by a sublattice of maximal rank. Notice that complex tori are usually presented as the quotient of a complex vector space by a full-rank additive lattice.

A complex torus (resp. tropical torus) is an abelian variety if it is given the choice of a polarization, which is a line bundle with positive integer Chern class. In the complex case, this condition is equivalent to the Riemann bilinear relation from \cite{griffiths2014principles}, which gives a criteria for the existence of a positive integer $(1,1)$-class. In the tropical setting, this condition translates to the existence of an element $C\in\Lambda\otimes N$ such that the product $CS^T\in (N\otimes N)_\RR$ is symmetric. Once such a choice has been made, it is possible to study curves in the class $C$.

\paragraph{Enumerative geometry and Gromov-Witten invariants.} Given an integer class $C$ that contains curves, it is natural to try to count curves in the class that satisfy certain given conditions. Such conditions are usually fixing the genus, and some geometric conditions such as passing through points. For instance, in the projective plane $\CC P^2$, look at degree $d$ and genus $g$ curves  that pass through $3d-1+g$ points. To be able to consider such enumerative problems, one has to know the dimension of the moduli space of curves so that we impose the right number of constraints. In the case of abelian surfaces, the dimension of the moduli space of genus $g$ curves in a given class $C$ is $g$, meaning we have to impose $g$ constraints to expect a finite number of curves.

For curves in abelian surfaces, there are two main enumerative problem that we can consider:
\begin{itemize}[label=$\circ$]
\item How many genus $g$ curves in the class $C$ pass through $g$ points ?
\item How many genus $g$ curves in a fixed linear system pass through $g-2$ points ?
\end{itemize}
In the second case, fixing the linear system imposes a codimension $2$ condition because the group of line bundles on an abelian surface has dimension $2$. In this paper, we are mainly interested to the first enumerative problem. The second enumerative problem is adressed in the second paper of the series. 

Sometimes, and it is here the case, the number of solutions to an enumerative problem does not depend on the choice of the constraints as long as it is generic. The result is thus called an \textit{invariant}. In fact, as proved in \cite{bryan1999generating}, these enumerative invariants coincide with Gromov-Witten invariants. The latter are defined as integrals of cohomology classes in the moduli space of curves inside the ambiant variety over some virtual fundamental class in the moduli space of curves. By integrating different cohomology classes, we see that the enumerative invariants are part of a much bigger family of invariants, not all of them having an enumerative interpretation. For instance, one can integrate $\lambda$-class or $\psi$-class. The invariants with integration of a $\lambda$-class might be related to the refined invariants introduced in this paper, as it is the case for toric surfaces \cite{bousseau2019tropical}.

Despite the enumerative counts being invariant, their computation often remains a challenge. The enumerative geometry of curves inside abelian surfaces has already been studied by J. Bryan and N. Leung \cite{bryan1999generating}, and J. Bryan, B. Oberdieck, R. Pandharipande and Q. Yin \cite{bryan2018curve}, although in these cases, the authors restricted to the case of primitive classes, which are classes $C$ that cannot be expressed as a multiple of a smaller class.

\paragraph{Tropical geometry and Correspondence Theorem.} As the numbers we are looking for do not depend on the choice of the constraints, it is now a standard approach to try to compute them close to the so-called \textit{tropical limit}. This was first done in the case of Severi degrees of toric varieties by G. Mikhalkin in \cite{mikhalkin2005enumerative}. Close to the tropical limit, the curves solution to the enumerative problems break into several components whose structure is encoded in objects called \textit{tropical curves}. Tropical curves are finite graphs whose edges have integer slope and satisfy a balancing condition. Then, a suitable \textit{correspondence theorem} allows one to recover the solutions to the enumerative problem close to the tropical limit from the tropical curves. The computation of the classical invariants has thus been reduced to a tropical enumerative problem, and the result is obtained by counting the tropical solutions with suitable multiplicities.

For the case of abelian surfaces, T. Nishinou gives a correspondence theorem in \cite{nishinou2020realization} for genus $g$ curves passing through $g$ points inside an abelian surface. One of the main results of \cite{nishinou2020realization} consists in providing the multiplicity $m_\Gamma^\CC$ with which to count the tropical solutions, so that their count gives the number of complex genus $g$ curves in a fixed class $C$ passing through a generic point configuration inside a complex abelian variety. The multiplicity provided in \cite{nishinou2020realization} lacks an easy expression such as the one from the toric setting in \cite{mikhalkin2005enumerative}, and the tropical enumerative problem has yet to be resolved.

In the setting of toric surfaces, the multiplicity provided by the correspondence theorem from \cite{mikhalkin2005enumerative} expresses as a product over the vertices of a tropical curve. In different settings, such as for instance in dimension bigger than $3$, the generalizations of the  correspondence theorem \cite{nishinou2006toric} \cite{mandel2020descendant} express the tropical multiplicities as the determinant of some map naturally associated to each tropical curve. In the case of abelian surfaces, the correspondence theorem from \cite{nishinou2020realization} does not express the multiplicity as a determinant because unlike all the cases previously considered, all the tropical curves are \textit{superabundant}, meaning the dimension of their deformation space does not match the expected dimension: they vary in dimension $g$ while the expected dimension is $g-1$. The multiplicity is thus defined a little bit differently, and though it can still be expressed as a determinant, this changes its computation. See section \ref{section correspondence theorem} for more details.

\paragraph{Refined invariants.} In the case of toric surfaces, the multiplicity provided by the correspondence theorem from \cite{mikhalkin2005enumerative} is expressed as a product over the vertices of the tropical curve. In \cite{block2016refined}, F. Block and L. G\"ottsche proposed to replace the vertex multiplicities by their quantum analog: an integer $a$ is replaced by $[a]_q=\frac{q^{a/2}-q^{-a/2}}{q^{1/2}-q^{-1/2}}$. This new multiplicity is a symmetric Laurent polynomial in a new variable $q$. Surprisingly, in many situations, the count of tropical curves solution to a suitable enumerative problem using this new multilicity happens not to depend either on the choice of the connstraints, as proven by I. Itenberg and G. Mikhalkin \cite{itenberg2013block} in the case of tropical curves in toric surfaces. This may not be expected because these multiplicities are not provided by a correspondence that guarantees that the count of tropical curves solutions to an enumerative problem leads to an invariant. The meaning of these refined invariants in classical geometry is thus a natural question. Such refined invariants have since been extended to many other situations. See for instance \cite{gottsche2019refined}, \cite{shustin2018refined}, \cite{blechman2019refined}, \cite{schroeter2018refined}, \cite{blomme2021refinedtrop}. The setting of floor diagrams has also been adapted to enable the computation of some of these refined invariants in \cite{block2016fock}.

These refined invariants were first defined purely tropically, and they possess a curious relationship to various other invariants in classical geometry. One of their conjectured interpretation is that they should coincide, in a certain sense, with the refinement of the Euler characteristic of some relative Hilbert scheme by its Hirzebruch genus \cite{gottsche2014refined}. Yet, other interpretations have since been found. In some situations, they have been proven to coincide with real refined classical invariants, obtained by counting real oriented rational curves according to the value of some quantum index \cite{mikhalkin2017quantum}, \cite{blomme2021refinedreal}. In other situations, P. Bousseau proved that through the change of variable $q=e^{iu}$, their generating series is connected to the generating series of Gromov-Witten invariants with the integration of a $\lambda$-class \cite{bousseau2019tropical}. The relation between these different approaches remains unclear. Yet, one common point in the last approaches is that they both need to cancel the denominators of the vertex multiplicities, suggesting that the refined multiplicity should in fact just be $\prod (q^{m_V/2}-q^{-m_V/2})$.

\subsection{Results}

The present paper provides new results in several directions. Shortly, we give a concrete formula to compute the complex multiplicity of a parametrized tropical curve provided by Nishinou's correspondence theorem \cite{nishinou2020realization}, we extend the setting of refined invariants to the case of tropical abelian surfaces. Concrete computations are enabled by a \textit{pearl diagram} algorithm presented in the third paper.

	\subsubsection{Multiplicity of the tropical curves.}

We consider a genus $g$ trivalent parametrized tropical curve $h:\Gamma\to\TT A$ (see Section \ref{section tropical curves} for definitions) passing through a generic configuration of $g$ points, where $\TT A=N_\RR/\Lambda$ is a tropical torus. We define the gcd of the curve $\delta_\Gamma$ as the gcd of the weights of the edges. See section \ref{section tropical curves} for more details. The following theorem gives a concrete formula to compute the complex multiplicity provided by Nishinou's correspondence theorem \cite{nishinou2020realization}.

\begin{theom}\ref{theorem multiplicity product}
The multiplicity $m_\Gamma^\CC$ provided by Nishinou's correspondence theorem splits as follows:
$$m_\Gamma^\CC=\delta_\Gamma\prod_{V\in V(\Gamma)}m_V,$$
where the vertex multiplicity $m_V=|\det(a_V,b_V)|$ is the determinant of two out of the three outgoing slopes, and $\delta_\Gamma$ is the gcd of the tropical curve $\Gamma$.
\end{theom}

The appearance of the gcd of the tropical curve may come as a surprise. However, it can be expected by considering an argument of homogeneity in the definition given of \cite{nishinou2020realization}.

The presence of the gcd of the tropical curve means that the complex multiplicity involves some global information on the curve. This does not prevent from computing the multiplicity of a unique tropical curve but might bring complications when trying to count the tropical curves passing through $g$ points because one has to keep track of this gcd.

The expression as a product over the vertices of the tropical curve suggests that a refinement as proposed by F. Block and L. G\"ottsche \cite{block2016refined} should also provide tropical refined invariants in this situation. We thus introduce the refined multiplicity of a tropical curve to be
$$m_\Gamma^q=\prod_{V\in V(\Gamma)}[m_V]_q=\prod_V \frac{q^{m_V/2}-q^{-m_V/2}}{q^{1/2}-q^{-1/2}}\in\ZZ[q^{\pm 1/2}].$$

	\subsubsection{Tropical refined invariants}

We give an invariance statement for the count of tropical curves of genus $g$ in a class $C$ passing through a configuration $\P$ of $g$ points using the previous multiplicities. In fact, one has an even more refined invariance by only considering the curves having a fixed gcd. This is due to the fact that through the deformations induced by the moving of the point configuration $\P$, the gcd of the tropical curves is preserved. In other words, all the tropical curves passing through the some point configuration split in different groups according to the value of their gcd, and we have invariance in each group. Thus, it is possible to count tropical curves with a fixed gcd and partially get rid of the factor $\delta_\Gamma $ in the complex multiplicity. We thus introduce the following counts of tropical curves:
\begin{align*}
N_{g,C,k}(\TT A,\P) & = \sum_{\substack{h(\Gamma)\supset\P \\ \delta_\Gamma =k}} m_\Gamma \in\NN, \\
BG_{g,C,k}(\TT A,\P) & = \sum_{\substack{h(\Gamma)\supset\P \\ \delta_\Gamma =k}} m^q_\Gamma \in \ZZ[q^{\pm 1/2}],\\
M_{g,C}(\TT A,\P) &  = \sum_{h(\Gamma)\supset\P} m_\Gamma = \sum_{k|\delta(C)} N_{g,C,k}(\TT A,\P) \in\NN , \\
N_{g,C}(\TT A,\P) &  = \sum_{h(\Gamma)\supset\P} \delta_\Gamma m_\Gamma = \sum_{k|\delta(C)}k N_{g,C,k}(\TT A,\P) \in\NN , \\
R_{g,C}(\TT A,\P) & = \sum_{h(\Gamma)\supset\P} \delta_\Gamma m^q_\Gamma = \sum_{k|\delta(C)}k BG_{g,C,k}(\TT A,\P) \in \ZZ[q^{\pm 1/2}] , \\
BG_{g,C}(\TT A,\P) & = \sum_{h(\Gamma)\supset\P} m^q_\Gamma = \sum_{k|\delta(C)} BG_{g,C,k}(\TT A,\P) \in \ZZ[q^{\pm 1/2}] . \\
\end{align*}

All the counts are over irreducible tropical curves $h:\Gamma\to\TT A$ of genus $g$ in the class $C$ that pass through the point configuration $\P$. The first two counts are over the curves that have a fixed gcd $k$. The other counts are over all the tropical curves (without gcd constraints) and can thus be expressed as a linear combination of the first two ones. The first count is obtained by specializing the second one at $q=1$. Thus, an invariance for the second count yields an invariance for all the others. We have the following invariance statements. First, we have an invariance regarding the point configuration $\P$.

\begin{theom}\ref{theorem point invariance}
The refined count $BG_{g,C,k}(\TT A,\P)$ does not depend on the choice of $\mathcal{P}$ as long as it is generic.
\end{theom}

The corresponding invariant is denoted by $BG_{g,C,k}(\TT A)$. Then, we have an invariance concerning the choice of the tropical torus $\TT A$.

\begin{theom}\ref{theorem torus invariant}
The refined invariant $BG_{g,C,k}(\TT A)$ does not depend on $\TT A$ as long as $\TT A$ contains curves in the class $C$ and is chosen generically among them.
\end{theom}

The dependence of $BG_{g,C,k}(\TT A)$ in terms of $\TT A$ depends in fact on the lattice generated by realizable classes in the tropical torus. Choosing bases of $N$ and $\Lambda$, the tropical torus $N_\RR/\Lambda$ is defined by the $2\times 2$ real matrix of the inclusion $S:\Lambda\hookrightarrow N_\RR$, and the classes that are realizable are given by integer matrices $C$ such that $CS^T\in\S_2(\RR)$. If $C$ is fixed and $S$ chosen generically so that this condition is satisfied, there will be no other matrices satisfying the condition, and the multiples of $C$ will be the only classes realizable by curves in the tropical torus associated to $S$. If $S$ is not chosen generically, there may be classes $C_1$ and $C_2$ such that $C_1+C_2=C$.

Due to the fact that the dimension of the deformation space of genus $g$ tropical curves in any class is equal to $g$, among the genus $g$ curves passing through a point configuration $\P$, there are reducible curves having irreducible components of genera that add up to $g$, splitting the point constraints among the different components. When deforming the tropical torus $\TT A$, which means changing the matrix $S$, it is possible to deform the irreducible components whose class keeps being realizable. This is the case of the irreducible components whose class is proportional to $C$. However, if for a non-generic $S$ we have a decomposition of $C$ as a sum of classes that are not proportional to $C$, for instance $C=C_1+C_2$, and it is not possible to separately deform the components. In other words, when deforming $\TT A$, a family of irreducible curves might deform to a reducible curve. The reducible curve gives no contribution to $BG_{g,C}(\TT A)$, but its deformation contributes for nearby $\TT A$. See section \ref{section invariance surface} for more details. Thus, the invariant $BG_{g,C}(\TT A)$ is not the same for any $\TT A$.

All the invariants are defined purely tropically. The count $N_{g,C}$ is interesting because Nishinou's correspondence theorem \cite{nishinou2020realization} ensures that it matches the number of genus $g$ curves in the class $C$ passing through $g$ points for a complex abelian surface, number that we denote by $\N_{g,C}$. The meaning of the refined invariants remains open although one can probably adapt results from \cite{bousseau2019tropical} to show that they coincide with Gromov-Witten invariants with insertion of $\lambda$-classes. Moreover, due to the invariance for each class of curves with fixed gcd, we have many refined invariants since it is possible to take any linear combination of the first invariants. Maybe some other choices could have an interpretation in complex or real geometry.

\subsection{Plan of the paper}

The paper is organized as follows.
\begin{itemize}[label=-]
\item The second section defines tropical tori and a description of the families of complex tori that tropicalize to a given tropical torus. If one knows the description of a tropical torus as $\TT A=N_\RR/\Lambda$, that the degree is the data of a class $C\in\Lambda\otimes N$, and that a class $C$ is realizable if and only if $CS^T$ is symmetric, this section can be skipped in a first lecture.
\item In the third section, we give definitions for tropical curves inside tropical abelian surfaces, a way to relate them to tropical curves in the plane $N_\RR$, and we compute the dimension of their deformation space.
\item The fourth section is devoted to stating the invariance results of the paper: first we set the enumerative problems, then we recall Nishinou's correspondence theorem, and last we state the invariance results. All these results are proved in the fifth section.
\item Theorem \ref{theorem multiplicity product} relies on the computation of the multiplicity of a planar tropical curve and is recalled in Appendix.
\end{itemize}

\textit{Acknowledgments.} The author is grateful to Hannah Markwig for some helpful remarks and kindly agreeing to read the paper. Research is supported in part by the SNSF grant 204125.

\section{Tropical Abelian surfaces}

\subsection{Tropical tori}

We follow partially the notations from \cite{halle2017tropical} and use the following definition of a tropical torus. In all what follows, let $N$ be a rank $2$ lattice, $M=\hom(N,\ZZ)$ its dual lattice. For a lattice $N$, we denote by $N_\RR$ the real vector space $N\otimes\RR$, and similarly $N_{\CC^*}=N\otimes\CC^*$.

\begin{defi}
A tropical torus $\TT A$ is a quotient
$$\TT A=N_\RR/\Lambda,$$
where $\Lambda\subset N_\RR$ is a lattice.
\end{defi}

There are two lattices. The lattice $N$ gives the tropical structure, and the lattice $\Lambda$ gives the period of the tropical torus $\TT A$. The lattice $N$ lives in fact in the tangent bundle of $\TT A$, which is the trivial bundle with fiber $N_\RR$. It is possible to fix a basis of $N$ and define $\TT A$ just as a quotient $\RR^2/\Lambda$, but this definition does not handle a change of basis, while the above definition does. Moreover, it provides a more natural definition of the dual torus.

\medskip

From now on, let $\TT A=N_\RR/\Lambda$ be a tropical torus define by two lattices $N$ and $\Lambda$, with a map $S:\Lambda\to N_\RR$. Choosing bases of the lattices $N$ and $\Lambda$, the inclusion $S$ amounts to the choice of a $2\times 2$ matrix with real coefficients. This choice is defined up to the change of basis, \textit{i.e.} the action of $SL_2(\ZZ)$ by multiplication on the left and on the right.

\medskip

\paragraph{Homology groups} The universal cover of $\TT A$ is $N_\RR$, and one has thus a natural identification of $H_1(\TT A,\ZZ)$ with $\Lambda$.

\paragraph{1-forms and cycles} As the tangent bundle $T(\TT A)$ is the trivial bundle with fiber $N_\RR$, the cotangent bundle $T^*(\TT A)$ is also trivial with fiber $M_\RR$, the dual of $N_\RR$, that contains the dual lattice $M$. The sections of the cotangent bundles with values in $M$ consist in the $1$-forms that take integral values on the lattice $N$.

\medskip

Given a $1$-form, \textit{i.e.} an element of $H^0(\TT A,T^*(\TT A))$, we can integrate it on cycles of $H_1(\TT A,\ZZ)\simeq\Lambda$. We choose to restrict to the tropical $1$-forms, \textit{i.e.} the ones that belong to $H^0(\TT A,M)$. The integration of tropical $1$-forms along cycles
$$(\alpha,\gamma)\in M\times\Lambda\longmapsto\int_\gamma\alpha \in\RR,$$
corresponds exactly to the map
$$(m,\lambda)\in M\times\Lambda\longmapsto \langle m,S\lambda\rangle\in \RR,$$
given by the inclusion $S:\Lambda\hookrightarrow N_\RR$, and the evaluation pairing $\langle,\rangle:M_\RR\times N_\RR\to\RR$. In other words, let $(\gamma_1,\gamma_2)$ be a basis of $\Lambda$, and $(m_1,m_2)$ a basis of $M$, with dual basis $(e_1,e_2)$ of $N$. The elements $m_1,m_2$  are coordinate functions on $N_\RR$. The matrix of the inclusion $S:\Lambda\hookrightarrow N_\RR$ in the chosen basis has elements that correspond to the integral of the $1$-forms $m_1,m_2\in M\simeq H^0(\TT A,M)$ over the cycles $\gamma_1,\gamma_2$. This way, the lattice $\Lambda$ appears as the period of $\TT A$.

\subsection{Tropical homology and intersection form}

In this section, we compute the tropical homology groups of a tropical torus. For a precise definition of the tropical homology groups, see \cite{itenberg2019tropical}. In our case, the cosheafs $\F_p$ used to compute tropical homology groups are just constant coefficients, with $\F_1=N$ and $\F_2=\Lambda^2 N$, the tropical homology groups $H_{p,q}(\TT A)=H_q(\TT A,\F_p)$ of $\TT A$ are as follows:
$$\begin{matrix}
& & & \Lambda^2 N\otimes\Lambda^2 \Lambda & & \\
& & \Lambda^2 N\otimes\Lambda & & N\otimes\Lambda^2\Lambda & \\
H_{p,q}(\TT A)=& \Lambda^2 N  & & \Lambda\otimes N & & \Lambda^2\Lambda \\
& & N & & \Lambda & \\
&  & & \ZZ & & \\
\end{matrix}.$$
We only care about the middle group $H_{1,1}(\TT A)\simeq\Lambda\otimes N$. We now choose orientations on the lattices $N$, $M$ and $\Lambda$, in a compatible way: orientations of $\Lambda^2N$ and $\Lambda^2 M$ determine each other, and the orientation of $N_\RR$ restricts to the orientation of $\Lambda$. This defines an intersection product on $H_{1,1}(\TT A)$. This group can be seen as the standard homology group of $\TT A$ but with coefficients in $N$. The intersection number between two $N$-cycles is obtained by summing over the intersection points the intersection number between their coefficients. In other words,
$$(\lambda\otimes n)\cdot(\lambda'\otimes n')=\det(\lambda,\lambda ')\det(n,n').$$

\begin{rem}
In our case, the tropical homology groups can be seen as the homology group with value in the homology of a torus $N_{\RR/\ZZ}=N\otimes\RR/\ZZ$. Thus, they are the homology groups of some $N_{\RR/\ZZ}$-principal bundle over $N_\RR/\Lambda$. Topologically, this space is a $4$-dimensional torus, which is just the associated complex variety, as depicted in the next section.
\end{rem}

\subsection{Degeneration of complex family to a tropical tori}
\label{section degeneration}

\subsubsection{Abelian surfaces in multiplicative form.} We now briefly describe degenerations of a family of complex tori. A $2$-dimensional complex torus is the quotient of a $2$-dimensional complex vector space by some rank $4$ lattice. A complex torus thus has a representation of the form $\CC A=\CC^2/\Lambda_\mathrm{add}$, where $\Lambda_\mathrm{add}$ is the image of a matrix $\Omega\in\M_{2,4}(\CC)$. Up to a change of coordinate inside the lattice $\Lambda_\mathrm{add}$ and the vector space $\CC^2$, one can always assume that the matrix $\Omega$ is of the form
$$\Omega=\begin{pmatrix}
I_2 & Z \\
\end{pmatrix}.$$
Then, letting $N$ be the sublattice spanned by the first two columns, $\CC^2$ appears as the vector space $N_\CC=N\otimes\CC$. Then, taking the exponential map, one can also give the following multiplicative presentation of the $2$-dimensional torus:
$$\CC A=N_\CC/\Lambda_\mathrm{add} \longrightarrow N_{\CC^*}/\Lambda_\mathrm{mult},$$
where the map takes $n\otimes z\in N_\CC$ to $e^{2i\pi nz}\in N_{\CC^*}$, which is the exponential taken coordinate by coordinate. This is well-defined on $\CC A$ if one quotients by the image of the span of the last two column vectors under the exponential map.

\begin{rem}
The downfall of the multiplicative notation is that we implicitly fix a preferred sublattice of $\Lambda_\mathrm{add}$, and change of basis of the lattice, as well as the action of $GL_2(\CC)$ are not that clear anymore.
\end{rem}

\subsubsection{Line bundles on abelian surfaces} Up to translation, line bundles on abelian surfaces are classified by their Chern class $c_1(\L)\in H^2(\CC A,\ZZ)$. We care about the line bundles that have sections. For such line bundles, the Chern class lies in fact in the Hodge cohomology group $H^{1,1}(\CC A)$, and pairs positively with the K\"ahler form on $\CC A$. According to \cite{griffiths2014principles}, a integer skew-symmetric matrix $Q$ is the Chern class of a positive line bundle if and only if
$$\left\{ \begin{array}{l}
\Omega Q^{-1}\Omega^T=0,\\
-i\Omega Q^{-1}\overline{\Omega}^T>0.\\
\end{array}\right.$$
Given an integer skew-symmetric matrix, it is always possible to choose a basis of the lattice such that the matrix has the following form:
$$Q=\begin{pmatrix}
0 & -c \\
c^T & 0 \\
\end{pmatrix}.$$
In our previous notations, this matrix $Q$ can be seen as an application $N\oplus\Lambda\to M\oplus\Lambda^*$, and $c:\Lambda\to M$. In fact, one can even assume that the matrix $c$ is of the form $\begin{pmatrix}
\delta_1 & 0 \\
0 & \delta_2 \\
\end{pmatrix}$, where $\delta_1 | \delta_2$. As $Q^{-1}=\begin{pmatrix}
0 & c^{-1T} \\
-c^{-1} & 0 \\
\end{pmatrix}$, the Riemann bilinear condition then becomes
$$\left\{\begin{array}{l}
Zc^{-1} \text{ is symmetric,}\\
\im(-Zc^{-1}) \text{ is positive definite.}\\
\end{array}\right.$$

Notice that as $c^{-1}:M\to\Lambda$ and $Z:\Lambda\to N$, the product has sense and is a map $M\to N$ that can be symmetric as a bilinear form on $M$. In the complex setting, curves are section of a positive line bundle, and the homology class realized by a curve is Poincar\'e dual to the Chern class of the line bundle. The Riemann bilinear relation satisfied by the Chern classes of line bundles imposes a constraint on the classes $C\in H_2(\CC A,\ZZ)$ realized by complex curves.

\subsubsection{Deformations.} We now consider a deformation of the abelian surface, which means a deformation of the lattice, so that the Riemann bilinear relation keeps being satisfied. We keep the sublattice $N$ fixed, so that we only have a deformation of $Z$. The deformation is taken of the form $Z_t=\frac{1}{2i\pi}(A+S\log t)$, where $A$ is some complex matrix, and $S$ is some integer matrix. The Riemann bilinear relations become $Ac^{-1}\in\S_2(\CC)$ and $Sc^{-1}\in\S_2^{++}(\RR)$.
As $\log t$ goes to infinity when $t$ goes to infinity, one can forget about the positiveness condition for $A$ since the condition for $\frac{1}{2i\pi}(A+S\log t)$ is satisfied close to the limit. In the exponential notations, our family of tori $\CC A_t$ is of the form
$$\CC A_t=N_{\CC^*}/\langle e^At^S \rangle,$$
which is called a Mumford family. The vectors spanning the lattice are $\lambda_j=(e^{a_{ij}}t^{s_{ij}})_i$.

\begin{rem}
It is also possible to consider families of tori that vary in a more complicated manner, taking general formal series as coefficients on the multiplicative side for instance. Moreover, it could seem natural to first consider some Mumford family, \textit{i.e.} $\Omega_t=\begin{pmatrix}
I_2 & \frac{1}{2i\pi}(A+S\log t) \\
\end{pmatrix}$, and then find the Chern classes of positive line bundles that survive the deformation. Unfortunately, such classes are not necessarily of the antidiagonal form. This is due to the following fact. Unlike the real setting, any lagrangian sublattice of a lattice with a skew-symmetric form does not necessarily have a lagrangian complement. When assuming that $Q$ is of a given form, we pick a decomposition of the lattice as a sum of two lagrangians sublattice. The deformation is considered afterwards. When first fixing the family, some sublattice is fixed, and this one does not necessarily have a lagrangian complement. One could change the lattice, but the deformation would not be of Mumford type anymore.
\end{rem}

In the multiplicative notation, for complex tori, we have an exact sequence
$$0\to \ker\pi \to N_{\CC^*}/\Lambda_t\xrightarrow{\pi=|\bullet|} N_\RR/\log|\Lambda_t|\to 0.$$
This exact sequence expresses $\CC A_t$ as a torus fibration of fiber $N_{\RR/\ZZ}$, over the base $N_\RR/\log|\Lambda_t|$ which is also a torus. Moreover, it splits at the topological level. In fact, this decomposition just emphasizes the decomposition of the first homology group of the associated complex surface:
$$H_1(\CC A_t,\ZZ)=N\oplus \Lambda_t.$$
The first part corresponds to the homology of the torus fiber, and the second to the homology of the base of the fibration. Notice that we recover the tropical homology groups of $\TT A$.

We also have the corresponding decomposition for cohomology:
$$H^1(\CC A_t,\ZZ)=M\oplus \Lambda^*,$$
where $M=N^*$. It gives the following decomposition of the $H^2(\CC A_t,\ZZ)$:
$$H^2(\CC A_t,\ZZ)=\Lambda^2\Lambda^*\oplus M\otimes\Lambda^*\oplus\Lambda^2 M.$$

\begin{rem}
Notice that this decomposition does not match the Hodge decomposition of $\CC A_t$, but it matches the decomposition provided by tropical cohomology. In fact, the map $c$ is the tropical Chern class of the tropical line bundle on the tropical abelian surface $\TT A$ obtained as limit of the complex line bundle. See second paper for more details.
\end{rem}

Let $\Lambda$ be the image of $S$. The tropicalization of the family $\CC A_t$ is the tropical torus $\TT A=N_\RR/\Lambda$. In \cite{nishinou2020realization}, T. Nishinou constructs degeneration of the family of abelian surfaces to a union of toric surfaces glued along their toric boundary as follows: given a rational  polyhedral decomposition of $\TT A$, one constructs a periodic fan $\Sigma$ inside $N_\RR\times\RR$, by taking the cone over the periodization of the decomposition. This fan gives an almost toric variety, which fibrates over $\CC$. The quotient by the $N\simeq\ZZ^2$ action gives the degeneration. The author then gives a correspondence theorem for tropical curves inside $\TT A$, and families of classical curves inside $\CC A_t$.

\section{Tropical curves in abelian surfaces}
\label{section tropical curves}

\subsection{Parametrized tropical curves}

\subsubsection{Abstract tropical curves.} An \textit{abstract tropical curve} is a finite metric graph $\Gamma$ that may possess several edges of infinite length, called \textit{unbounded ends}, which all need to be adjacent to univalent vertices. The number of neighbors of a vertex is called its \textit{valency}. Its genus $g$ is equal to its first Betti number, \textit{i.e.} $b_1(\Gamma)=g$. The vertices have no genus. The set of edges is denoted by $E(\Gamma)$ and its set of vertices by $V(\Gamma)$. The curve is said to be trivalent if every vertex has three neighbors. The length of an edge $e$ is denoted by $l_e$. An isomorphism between two abstract tropical curves is an isometry. A \textit{marked abstract tropical curve} is an abstract tropical curve with the choice of $n$ points on it, labeled by $[\![1;n]\!]$. By refining the graph structure, we can assume that the marked points are vertices of the graph, which are generically bivalent.

\subsubsection{Parametrized tropical curves.} We now get to parametrized tropical curves. Let $\TT A=N_\RR/\Lambda$ be a tropical torus.

\begin{defi}
A parametrized tropical curve inside a tropical torus $\TT A=N_\RR/\Lambda$ (resp. $N_\RR)$ is a map $h:\Gamma\to \TT A$ (resp. $N_\RR$) from an abstract tropical curve $\Gamma$ such that
\begin{itemize}[label=-]
\item $h$ is affine with slope in $N$ on each edge of $\Gamma$. For an oriented edge $e$, its slope $u_e$ is of the form $w_eu'_e$ where $u'_e\in N$ is a primitive vector, and $w_e$ a positive integer called the \textit{weight} of the edge.
\item $h$ satisfies the balancing condition: for each vertex $V\in V(\Gamma)$,
$$\sum_{e\ni V}w_eu_e=0,$$
when each $e$ is oriented with $V$ as its source.
\end{itemize}
\end{defi}

\begin{rem}
For a parametrized tropical curve in $\TT A$, the curve $\Gamma$ has no unbounded ends, while it has several for a curve inside $N_\RR$.
\end{rem}

\begin{figure}
\begin{center}
\begin{tabular}{ccc}
\begin{tikzpicture}[line cap=round,line join=round,>=triangle 45,x=0.3cm,y=0.3cm]
\clip(0,0) rectangle (10,10);
\draw [line width=0.5pt] (0,0)-- (8,2);
\draw [line width=0.5pt] (0,0)-- (2,8);
\draw [line width=0.5pt] (8,2)-- (10,10);
\draw [line width=0.5pt] (2,8)-- (10,10);

\draw [line width=1.5pt] (4,1)-- (4,4);
\draw [line width=1.5pt] (1,4)-- (4,4);
\draw [line width=1.5pt] (4,4)-- (6,6);
\draw [line width=1.5pt] (6,6)-- (6,9);
\draw [line width=1.5pt] (6,6)-- (9,6);

\begin{scriptsize}

\end{scriptsize}
\end{tikzpicture}
&
\begin{tikzpicture}[line cap=round,line join=round,>=triangle 45,x=0.3cm,y=0.3cm]
\clip(0,0) rectangle (14,11);
\draw [line width=0.5pt] (0,0)-- (12,3);
\draw [line width=0.5pt] (0,0)-- (2,8);
\draw [line width=0.5pt] (2,8)-- (14,11);
\draw [line width=0.5pt] (12,3)-- (14,11);

\draw [line width=1.5pt] (3,2)-- (4,3);
\draw [line width=1.5pt] (4,3)-- (4,5);
\draw [line width=1.5pt] (4,3)-- (6,3);
\draw [line width=1.5pt] (6,3)-- (8,5);
\draw [line width=1.5pt] (0.5,2)-- (3,2);
\draw [line width=1.5pt] (3,0.75)-- (3,2);
\draw [line width=1.5pt] (4,5)-- (5,6);
\draw [line width=1.5pt] (1.25,5)-- (4,5);
\draw [line width=1.5pt] (5,6)-- (5,8.75);
\draw [line width=1.5pt] (5,6)-- (9,6);
\draw [line width=1.5pt] (9,6)-- (11,8);
\draw [line width=1.5pt] (11,8)-- (13.25,8);
\draw [line width=1.5pt] (11,8)-- (11,10.25);
\draw [line width=1.5pt] (8,5)-- (8,9.5);
\draw [line width=1.5pt] (8,5)-- (12.5,5);
\draw [line width=1.5pt] (6,3)-- (6,1.5);
\draw [line width=1.5pt] (9,6)-- (9,2.25);

\begin{scriptsize}

\end{scriptsize}
\end{tikzpicture}
&
\begin{tikzpicture}[line cap=round,line join=round,>=triangle 45,x=0.3cm,y=0.3cm]
\clip(0,0) rectangle (10,10);
\draw [line width=0.5pt] (2,4)-- (0,10);
\draw [line width=0.5pt] (2,4)-- (10,0);
\draw [line width=0.5pt] (0,10)-- (8,6);
\draw [line width=0.5pt] (10,0)-- (8,6);

\draw [line width=2pt] (1.33333333333,6)-- (4,6);
\draw [line width=1.5pt] (4,6)-- (8,2);
\draw [line width=1.5pt] (4,6)-- (5.333333333,7.33333333);
\draw [line width=1.5pt] (8,2)-- (7.333333333,1.33333333);
\draw [line width=2pt] (8,2)-- (9.3333333333,2);

\begin{scriptsize}
\draw (2,7) node {$2$};
\draw (8.5,2.5) node {$2$};
\end{scriptsize}
\end{tikzpicture}
\\
$(a)$ & $(b)$ & $(c)$\\
\end{tabular}

\caption{\label{figure example tropical curves}Three examples of tropical curves in tropical tori. The first and third are of genus $2$, and the second of genus $5$.}
\end{center}
\end{figure}
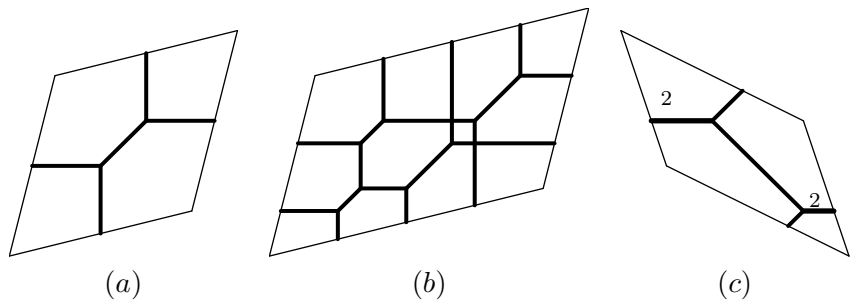

\begin{expl}
On Figure \ref{figure example tropical curves}, we can see three different examples of tropical curves in different tropical tori $\TT A$.
\end{expl}

\subsubsection{Degree of a tropical curve.} In particular, for each edge, the $1$-chain element with coefficients in $N$: $u_e e\in  C_1(\TT A,N)$ does not depend on the chosen orientation. Then, the curve realizes a $1$-chain
$$\sum_{e\in E(\Gamma)} u_e e\in C_1(\TT A,N).$$
The balancing condition ensures that this chain is in fact a cycle. Thus, it realizes a homology class $C=[\Gamma]\in H_1(\TT A,N)=\Lambda\otimes N=H_{1,1}(\TT A)$.

\begin{defi}
The class $C$ in $\Lambda\otimes N$ realized by a parametrized tropical curve is called the \textit{degree} of the curve.
\end{defi}

Given a curve $h:\Gamma\to \TT A$, its degree $C\in\Lambda\otimes N$ can be seen as a map $\Lambda^*\to N$, and thus as a matrix. The following proposition gives a description of its matrix.

\begin{prop}
Take a parallelogram that is a fundamental domain for $\TT A=N_\RR/\Lambda$. Let $(\lambda_1,\lambda_2)$ be the basis of $\Lambda$ formed by two edges of the parallelogram. Let $n_1$ be the sum of the slopes of the edges intersecting the side opposite to $\lambda_2$, oriented outward the fundamental domain, and let $n_2$ the sum for the edges intersecting the side opposite to $\lambda_1$. Then $C$ is the map that sends $\lambda_1^*$ to $n_1$ and $\lambda_2^*$ to $n_2$.
\end{prop}

\begin{proof}
Let $C'$ be the class defined in the proposition. We check that $C'$ has the same intersection numbers with a basis of $H_{1,1}(\TT A)$ as the curve $\Gamma$. Thus, they represent the same class.
\end{proof}


The classes realized by tropical curves are of a very specific kind: they satisfy the following property.

\begin{prop}
A class $C\in\Lambda\otimes N$ can be realized by a tropical curve if and only if $CS^T$ is symmetric.
\end{prop}

\begin{proof}
Notice that $S^T:M\to\Lambda^*$ and $C$ is a matrix of a map $\Lambda^*\to N$, so that the product has a meaning. If the condition is satisfied, it is easy to construct a genus $2$ curve in the class $C$. Conversely, assume that there is some tropical curve inside the surface. Choose a fundamental domain for $\TT A=N_\RR/\Lambda$. Cut the tropical curve inside the fundamental domain and replace the cut edges by unbounded ends to get a true tropical curve inside $N_\RR$. By tropical Menelaus theorem \cite{mikhalkin2017quantum}, the sum of the moments of the unbounded ends is $0$. The unbounded ends come in pairs of ends glued when passing from the fundamental domain to $\TT A$. The sum of their moments is equal to $\det(\lambda_1,n_e)$. Adding the contribution for all edges yields the result. See next section for more details.
\end{proof}

\begin{rem}
Fixing the lattices $N$ and $\Lambda$, the tropical torus is given by a $2\times 2$ real matrix $S$, corresponding to the inclusion $\Lambda\to N_\RR$. A class $C\in \Lambda\otimes N$ is realizable if and only if $S^T C$ is symmetric. In other words, the matrix $S$ gives a linear form on $(\Lambda\otimes N)_\RR$, and the set of realizable classes is some subset of the intersection between an hyperplane and the lattice $\Lambda\otimes N$. Generically, there will be none. Alternatively, the realizability of a given class $C$ imposes a codimension $1$ constraint on $S$.
\end{rem}

\begin{expl}
Choosing basis of the lattices $\Lambda$ and $N$, a class $N$ is also represented by a matrix, a map $\Lambda^*\to N$. Getting back to the curves on Figure \ref{figure example tropical curves}, we take as a basis of $N$ the canonical basis of the plane where the figures are drawn, and as a basis of $\Lambda$ the sides of the parallelogram, first the almost vertical one, and then the almost horizontal one.  we have the following:
\begin{enumerate}[label=$(\alph*)$]
\item The matrix of the inclusion $\Lambda\to N_\RR$ is of the form $S=\begin{pmatrix}
a & b \\
b & a \\
\end{pmatrix}$, and the curve represents the class $C=\begin{pmatrix}
1 & 0 \\
0 & 1 \\
\end{pmatrix}$. We easily check that $CS^T$ is symmetric.
\item We have $S=\begin{pmatrix}
12 & 2 \\
3 & 8 \\
\end{pmatrix}$, and the curve represents the class $C=\begin{pmatrix}
2 & 0 \\
0 & 3 \\
\end{pmatrix}$, and we have again $CS^T$ symmetric.
\item We have $S=\begin{pmatrix}
4 & -1 \\
-2 & 3 \\
\end{pmatrix}$, and the curve represents the class $C=\begin{pmatrix}
2 & 1 \\
0 & 1 \\
\end{pmatrix}$, and we have again $CS^T$ symmetric.
\end{enumerate}
Concretely, the matrix $C$ is the matrix whose columns are the sums of the slopes intersecting the right (resp. top) side of the parallelogram respectively. Meanwhile, $S$ is the matrix whose columns are the coordinate of the bottom side and left side of the parallelogram in the basis of $N_\RR$.
\end{expl}

\begin{expl}
Consider the tripod with directions $(-\alpha_1-\alpha_2,-\beta_1-\beta_2)$, $(\alpha_1,\beta_1)$ and $(\alpha_2,\beta_2)$ and  with respective lengths $a$, $b$ and $c$. Take $S=\begin{pmatrix}
a(\alpha_1+\alpha_2)+b\alpha_1 & a(\alpha_1+\alpha_2)+c\alpha_2 \\
a(\beta_1+\beta-2)+b\beta_1 & a(\beta_1+\beta_2)+c\beta_2 \\
\end{pmatrix}$, and $C=\begin{pmatrix}
\alpha_1 & \alpha_2 \\
\beta_1 & \beta_2 \\
\end{pmatrix}$. We can check that $C S^T$ is always symmetric. We indeed have a genus $2$ tropical curve in the class $C$, as depicted on Figure \ref{figure tripode}. Notice that the corner of the parallelogram is the second vertex of the curve. The example of Figure \ref{figure example tropical curves} $(a)$ is also of this type, except the curve has been translated so that the corner of the parallelogram is not a vertex of the curve anymore.
\end{expl}

\begin{figure}
\begin{center}
\begin{tabular}{ccc}
\begin{tikzpicture}[line cap=round,line join=round,>=triangle 45,x=0.6cm,y=0.6cm]
\clip(0,0) rectangle (6,6);
\draw [line width=0.5pt] (0,0)-- (0.5,5);
\draw [line width=0.5pt] (0,0)-- (4,1);
\draw [line width=0.5pt] (0.5,5)-- (4.5,6);
\draw [line width=0.5pt] (4,1)-- (4.5,6);

\draw [line width=1.5pt] (0,0)-- (2,2);
\draw [line width=1.5pt] (2,2)-- (4,1);
\draw [line width=1.5pt] (2,2)-- (0.5,5);
\begin{scriptsize}

\end{scriptsize}
\end{tikzpicture}
&
\begin{tikzpicture}[line cap=round,line join=round,>=triangle 45,x=0.4cm,y=0.4cm]
\clip(0,0) rectangle (10,10);
\draw [line width=0.5pt] (0,0)-- (3,6);
\draw [line width=0.5pt] (0,0)-- (6,2);
\draw [line width=0.5pt] (6,2)-- (9,8);
\draw [line width=0.5pt] (3,6)-- (9,8);

\draw [line width=1.5pt] (0,0)-- (3,2);
\draw [line width=2pt] (3,2)-- (3,6);
\draw [line width=2pt] (3,2)-- (6,2);
\begin{scriptsize}
\draw (5,2.5) node {$3$};
\draw (3.5,4) node {$2$};
\end{scriptsize}
\end{tikzpicture} \\
$(a)$ & $(b)$ \\
\end{tabular}
\caption{\label{figure tripode} Examples of genus $2$ tropical curve.}
\end{center}
\end{figure}
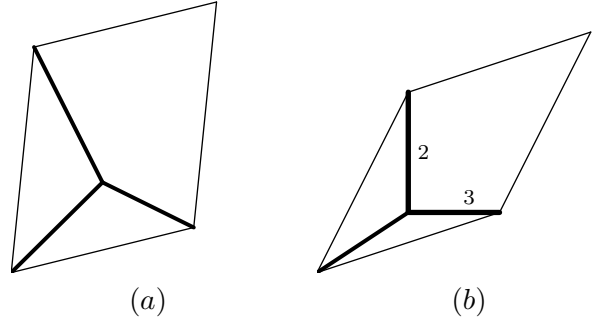

\subsubsection{Miscellaneous} We end with some definition that intervenes at different moments of the paper regarding the enumerative problems and multiplicities.

\begin{defi}
\begin{itemize}[label=-]
\item For an irreducible tropical curve $h:\Gamma\to\TT A$, its gcd $\delta_\Gamma$ is the gcd of the weights of its edges.
\item An irreducible curve is said to be \textit{primitive} if its gcd is $1$.
\item For an irreducible curve $h:\Gamma\to\TT A$ and an integer $k$, we have a tropical curve $k h:k\Gamma\to\TT A$ where the curve $k\Gamma$ is $\Gamma$ where the lengths of the edges have been divided by $k$, and the slope of $k h$ on the edges is the slope of $h$ multiplied by $k$. This curve is denoted by $k\Gamma$.
\end{itemize}
\end{defi}

In particular, $\delta_{k\Gamma}=k\delta_\Gamma$, and every irreducible parametrized tropical curve $h:\Gamma\to\TT A$ can be written as $\delta_\Gamma\Gamma'$ for a primitive parametrized tropical curve $\Gamma'$. For a reducible curve, we can speak about the gcd of each of its irreducible components.

\subsection{Cutting/lifting procedure}
\label{section lifting}

As the tropical torus $\TT A$ is constructed as a quotient of $N_\RR$, it is possible to relate tropical curves inside $\TT A$ to tropical curves inside $N_\RR$. We explicit the relation and procedure to construct one from another in this section.

\medskip

Let $h:\Gamma\to A$ be a connected genus $g$ parametrized tropical curve in the tropical abelian surface $\TT A$. Assume that it has no contracted edge. If $\Gamma$ had a contracted edge $e$, we would reduce to the following situations:
\begin{itemize}[label=-]
\item If $e$ is disconnecting, we have $\Gamma-e=\Gamma_1\sqcup\Gamma_2$ and we reparametrize $h:\Gamma\to \TT A$ by $\Gamma_1\sqcup\Gamma_2$, which is a reducible curve, with components of respective genera $g_1$ and $g_2$ with $g_1+g_2=g$.
\item If $e$ is not disconnecting, $\Gamma-e$ is a connected curve of genus $g-1$.
\end{itemize}

\begin{defi}
A subset $\Q\subset\Gamma$ of points located on the edges of $\Gamma$ is said to be \textit{lifting} if the image of the morphism $h_*:\pi_1(\Gamma\backslash\Q)\to \pi_1(\TT A)=\Lambda\simeq \ZZ^2$ is trivial.
\end{defi}

Taking a point on every edge of the complement of a spanning tree of $\Gamma$, we see that a lifting set always exists, and the complement $\Gamma\backslash\Q$ can be connected. The trivialness of the map $h_*:\pi_1(\Gamma\backslash\Q)\to \pi_1(\TT A)$ ensures that one can lift $h$ to
$$\tilde{h}:\Gamma\backslash\Q \longrightarrow N_\RR.$$
Then, replacing the edges that have been cut by infinite edges, we get a parametrized tropical curve
$$\tilde{h}:\widetilde{\Gamma}_\Q\longrightarrow N_\RR,$$
inside the plane $N_\RR$. When $\Q$ is fixed, we just denote it by $\widetilde{\Gamma}$.

\begin{expl}
On Figure \ref{figure lifting process} we see a tropical curve of genus $5$ in the class $\begin{pmatrix}
2 & 0 \\
0 & 2 \\
\end{pmatrix}$ in some tropical torus. We choose $5$ points on it so that the complement has no cycle. This choice is depicted on $(a)$. Then we unfold it to the universal cover and consider a connected component of the complement of the marked point. We get $(b)$. Finally, we prolong to infinity the edges cut by a point of the lifting set to get a tropical curve in $N_\RR$ on $(c)$. Notice that the unbounded end of the lifted curve are paired together.
\end{expl}

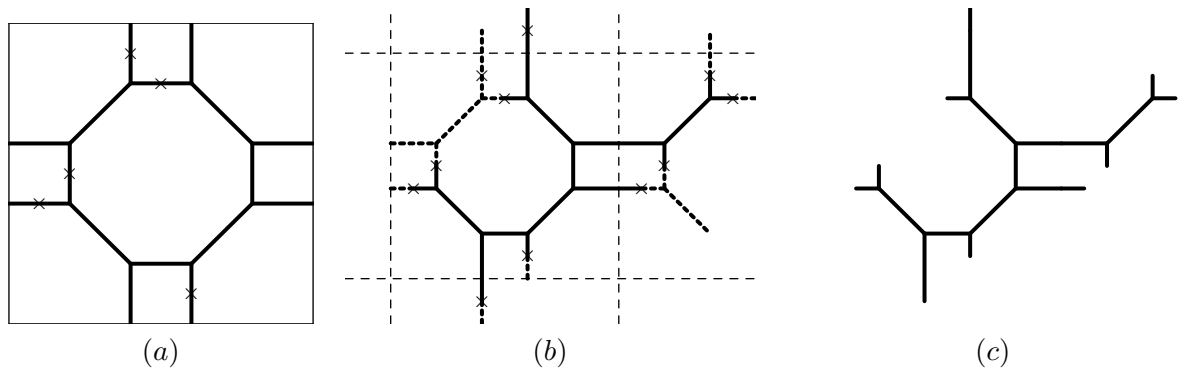
\begin{figure}
\begin{center}
\begin{tabular}{ccc}
\begin{tikzpicture}[line cap=round,line join=round,>=triangle 45,x=0.8cm,y=0.8cm]
\clip(0,0) rectangle (5,5);
\draw [line width=0.5pt] (0,0)-- (0,5);
\draw [line width=0.5pt] (0,0)-- (5,0);
\draw [line width=0.5pt] (0,5)-- (5,5);
\draw [line width=0.5pt] (5,0)-- (5,5);

\draw [line width=1.5pt] (0,2)-- (1,2);
\draw [line width=1.5pt] (1,2)-- (2,1);
\draw [line width=1.5pt] (2,1)-- (3,1);
\draw [line width=1.5pt] (3,1)-- (4,2);
\draw [line width=1.5pt] (4,2)-- (5,2);

\draw [line width=1.5pt] (0,3)-- (1,3);
\draw [line width=1.5pt] (1,3)-- (2,4);
\draw [line width=1.5pt] (2,4)-- (3,4);
\draw [line width=1.5pt] (3,4)-- (4,3);
\draw [line width=1.5pt] (4,3)-- (5,3);

\draw [line width=1.5pt] (1,3)-- (1,2);
\draw [line width=1.5pt] (2,0)-- (2,1);
\draw [line width=1.5pt] (3,0)-- (3,1);
\draw [line width=1.5pt] (4,3)-- (4,2);

\draw [line width=1.5pt] (2,4)-- (2,5);
\draw [line width=1.5pt] (3,4)-- (3,5);
\begin{scriptsize}
\draw (0.5,2) node {$\times$};
\draw (1,2.5) node {$\times$};
\draw (3,0.5) node {$\times$};
\draw (2.5,4) node {$\times$};
\draw (2,4.5) node {$\times$};
\end{scriptsize}
\end{tikzpicture}
&
\begin{tikzpicture}[line cap=round,line join=round,>=triangle 45,x=0.6cm,y=0.6cm]
\clip(-1,-1) rectangle (8,6);
\draw [line width=0.5pt,dashed] (0,-1)-- (0,6);
\draw [line width=0.5pt,dashed] (-1,0)-- (8,0);
\draw [line width=0.5pt,dashed] (-1,5)-- (8,5);
\draw [line width=0.5pt,dashed] (5,-1)-- (5,6);

\draw [line width=1.5pt,dotted] (0,2)-- (0.5,2);
\draw [line width=1.5pt] (0.5,2)-- (1,2);
\draw [line width=1.5pt] (1,2)-- (2,1);
\draw [line width=1.5pt] (2,1)-- (3,1);
\draw [line width=1.5pt] (3,1)-- (4,2);
\draw [line width=1.5pt] (4,2)-- (5,2);

\draw [line width=1.5pt,dotted] (0,3)-- (1,3);
\draw [line width=1.5pt,dotted] (1,3)-- (2,4);
\draw [line width=1.5pt,dotted] (2,4)-- (2.5,4);
\draw [line width=1.5pt] (2.5,4)-- (3,4);
\draw [line width=1.5pt] (3,4)-- (4,3);
\draw [line width=1.5pt] (4,3)-- (5,3);

\draw [line width=1.5pt,dotted] (1,3)-- (1,2.5);
\draw [line width=1.5pt] (1,2.5)-- (1,2);
\draw [line width=1.5pt] (2,-0.5)-- (2,1);
\draw [line width=1.5pt,dotted] (2,-0.5)-- (2,-1);
\draw [line width=1.5pt,dotted] (3,0)-- (3,0.5);
\draw [line width=1.5pt] (3,0.5)-- (3,1);
\draw [line width=1.5pt] (4,3)-- (4,2);

\draw [line width=1.5pt,dotted] (2,4)-- (2,5.5);
\draw [line width=1.5pt] (3,4)-- (3,5.5);
\draw [line width=1.5pt] (3,5.5)-- (3,6);

\draw [line width=1.5pt] (5,2)-- (5.5,2);
\draw [line width=1.5pt,dotted] (5.5,2)-- (6,2);
\draw [line width=1.5pt] (6,3)-- (6,2.5);
\draw [line width=1.5pt,dotted] (6,2.5)-- (6,2);
\draw [line width=1.5pt] (5,3)-- (6,3);
\draw [line width=1.5pt] (6,3)-- (7,4);
\draw [line width=1.5pt] (7,4)-- (7.5,4);
\draw [line width=1.5pt,dotted] (7.5,4)-- (8,4);
\draw [line width=1.5pt,dotted] (6,2)-- (7,1);
\draw [line width=1.5pt] (7,4)-- (7,4.5);
\draw [line width=1.5pt,dotted] (7,4.5)-- (7,5.5);

\begin{scriptsize}
\draw (0.5,2) node {$\times$};
\draw (1,2.5) node {$\times$};
\draw (3,0.5) node {$\times$};
\draw (2.5,4) node {$\times$};
\draw (2,4.5) node {$\times$};

\draw (5.5,2) node {$\times$};
\draw (6,2.5) node {$\times$};
\draw (3,5.5) node {$\times$};
\draw (7.5,4) node {$\times$};
\draw (7,4.5) node {$\times$};
\draw (2,-0.5) node {$\times$};
\end{scriptsize}
\end{tikzpicture}
&
\begin{tikzpicture}[line cap=round,line join=round,>=triangle 45,x=0.6cm,y=0.6cm]
\clip(-1,-1) rectangle (8,6);

\draw [line width=1.5pt] (0.5,2)-- (1,2);
\draw [line width=1.5pt] (1,2)-- (2,1);
\draw [line width=1.5pt] (2,1)-- (3,1);
\draw [line width=1.5pt] (3,1)-- (4,2);
\draw [line width=1.5pt] (4,2)-- (5,2);

\draw [line width=1.5pt] (2.5,4)-- (3,4);
\draw [line width=1.5pt] (3,4)-- (4,3);
\draw [line width=1.5pt] (4,3)-- (5,3);

\draw [line width=1.5pt] (1,2.5)-- (1,2);
\draw [line width=1.5pt] (2,-0.5)-- (2,1);

\draw [line width=1.5pt] (3,0.5)-- (3,1);
\draw [line width=1.5pt] (4,3)-- (4,2);

\draw [line width=1.5pt] (3,4)-- (3,5.5);
\draw [line width=1.5pt] (3,5.5)-- (3,6);

\draw [line width=1.5pt] (5,2)-- (5.5,2);

\draw [line width=1.5pt] (6,3)-- (6,2.5);

\draw [line width=1.5pt] (5,3)-- (6,3);
\draw [line width=1.5pt] (6,3)-- (7,4);
\draw [line width=1.5pt] (7,4)-- (7.5,4);

\draw [line width=1.5pt] (7,4)-- (7,4.5);

\end{tikzpicture} \\
$(a)$ & $(b)$ & $(c)$ \\
\end{tabular}
\caption{\label{figure lifting process}On the left a tropical curve in a tropical torus with a lifting set, and on the right the corresponding lifting in the plane $N_\RR$.}
\end{center}
\end{figure}

Let $x=|\Q|$ and let $n_1,\dots,n_x$ be the slopes of the edges on which lie the points in $\Q$. The curve $(\widetilde{\Gamma},\tilde{h})$ is of genus $g-x$ and has degree $\{n_i,-n_i\}_{i=1}^x$. Recall that the degree of a tropical curve inside $N_\RR$ is the collection of the slopes of its unbounded ends.

\begin{lem}
Let $i\in[\![1;x]\!]$ indexing a pair of ends $\{e_i,f_i\}$ of $\widetilde{\Gamma}$, corresponding to a point $p_i$ in $\Gamma$. Let $\tilde{\gamma}$ be a path from $e_i$ to $f_i$ inside $\widetilde{\Gamma}$. It projects to a loop $\gamma$ in $(\Gamma\backslash\Q)\cup\{p_i\}$, and thus in $\TT A$. Then, considering $\gamma\in H_1(\TT A)$, one has
$$\det(n_i,e_i)+\det(-n_i,f_i)=\det(n_i,S\gamma).$$
\end{lem}

\begin{proof}
The element $\det(n_i,-)\in M$ is a tropical $1$-form on $\TT A$. Its value on the loop $\gamma$ is precisely $\det(n_i,S\gamma)$. To compute it, we can also parametrize the loop $\gamma$ by the path $\tilde{\gamma}$:
$$\det(n_i,S\gamma)=\sum_k \det(n_i,l_ku_{e_k}),$$
where the path is comprised of the edges $e_k$, of slope $u_{e_k}$ and length $l_k$. This can be lifted to $N_\RR$, and by linearity of $\det(n_i,-)$, we get the announced result.
\end{proof}

\begin{rem}
For a tropical curve $\Gamma\to N_\RR$ and $e$ an unbounded end of $\Gamma$ of slope $n_e$, the evaluation of $\det(n_e,-)$ at any point of the end is called the \textit{moment} of the end. Using the balancing condition, we get the tropical Menelaus theorem, proved in \cite{mikhalkin2017quantum}. It asserts that the sum of the moments of the unbounded ends is $0$.
\end{rem}

Using the preceding remark, another way to reformulate the result is then that the moments of the unbounded edges of $\widetilde{\Gamma}$ is constrained by the tropical abelian surface $\TT A$.

\begin{prop}\label{proposition deformation}
A small deformation of $\Gamma$ lifts to a small deformation of $\widetilde{\Gamma}$. Conversely, a small deformation of $\widetilde{\Gamma}$ lifts to a deformation of $\Gamma$ if and only if it satisfies the conditions
$$\det(n_i,e_i)+\det(-n_i,f_i)=\det(n_i,S\gamma).$$
\end{prop}

\begin{proof}
The first part is obvious, since a lifting set keeps being lifting for a small deformation. Moreover, the preceding lemma ensures that, as $\TT A$ is fixed, the relation between the moments of unbounded ends keeps being satisfied through the deformation. Conversely, the relation between the moments ensures that projecting the curve from $N_\RR$ to $\TT A$, it is possible to glue them together and get a curve inside $\TT A$.
\end{proof}

\begin{expl}
On Figure \ref{figure deformation torus and plane} we depict a tropical curve in a tropical torus and a small deformation of the curve. The deformation clearly lifts to a small deformation of the lifted curve inside $N_\RR$, taking for instance as a lifting set the points of intersection with the boundary of the parallelogram.
\end{expl}

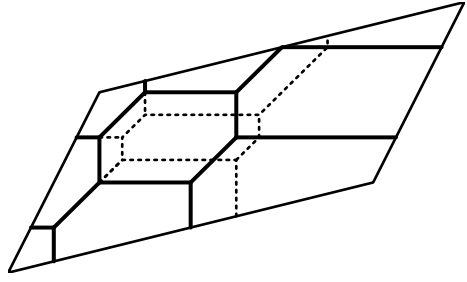
\begin{figure}
\begin{center}
\begin{tikzpicture}[line cap=round,line join=round,>=triangle 45,x=0.6cm,y=0.6cm]
\clip(0,0) rectangle (10,6);
\draw [line width=1pt] (0,0)-- ++(8,2)-- ++(2,4)-- ++(-8,-2)-- ++(-2,-4);
\draw [line width=1.5pt] (0.5,1)-- ++(0.5,0)-- ++(1,1)-- ++(2,0)-- ++(1,1)-- ++(3.5,0);
\draw [line width=1.5pt] (1.5,3)-- ++(0.5,0)-- ++(1,1)-- ++(2,0)-- ++(1,1)-- ++(3.5,0);
\draw [line width=1.5pt] (1,0.25)-- (1,1);
\draw [line width=1.5pt] (2,2)-- ++(0,1);
\draw [line width=1.5pt] (4,1)-- ++(0,1);
\draw [line width=1.5pt] (5,3)-- ++(0,1);
\draw [line width=1.5pt] (3,4)-- ++(0,0.25);

\draw [line width=1pt,dotted] (5,1.25)-- ++(0,1.25)-- ++(-2.5,0)-- ++(0,0.5)-- ++(0.5,0.5)-- ++(2.5,0)-- ++(1.5,1.5)-- ++(0,0.25);
\draw [line width=1pt,dotted] (2,2)-- ++(0.5,0.5);
\draw [line width=1pt,dotted] (2,3)-- ++(0.5,0);
\draw [line width=1pt,dotted] (3,4)-- ++(0,-0.5);
\draw [line width=1pt,dotted] (2,2)-- ++(0.5,0.5);
\draw [line width=1pt,dotted] (5,2.5)-- ++(0.5,0.5)-- ++(0,0.5);
\end{tikzpicture}
\caption{\label{figure deformation torus and plane}A small deformation of a tropical curve inside a tropical torus.}
\end{center}
\end{figure}

\begin{lem}
Let $h:\Gamma\to \TT A$ be a parametrized tropical curve. A disconnecting edge of $\Gamma$ has slope $0$ under $h$.
\end{lem}

\begin{proof}
Let $e$ be a disconnecting edge for $\Gamma$ and let $\Q$ be a lifting set for $\Gamma$ such that $\Gamma\backslash\Q$ is connected. As $e$ is disconnecting, it cannot contain a point of $\Q$. Moreover, as $e$ is disconnecting for $\Gamma$, it also is for $\widetilde{\Gamma}_\Q$. Finally, for any pair of ends $\{e_i,f_i\}$ associated to a point $p_i\in\mathcal{P}$ belong to the same component of $\widetilde{\Gamma}-e$, otherwise $e$ would not be disconnecting for $\Gamma$. Then, the balancing condition on either component of $\widetilde{\Gamma}-e$ ensures that $e$ has slope $0$.
\end{proof}

\subsection{Dimension of the moduli space}

Let $C\in H_{1,1}(\TT A,\ZZ)=\Lambda\otimes N$ be a class that is realizable in $\TT A$: $C S^T\in\S_2(\RR)$. Let $\M_{g,n}(\TT A,C)$ be the moduli space of genus $g$ irreducible tropical curves with $n$ marked points in the class $C$. It is a polyedral complex which is a union over the different combinatorial types of the tropical curves $\Gamma$. We assume that $\TT A$ is \textit{simple}: it does not contain any elliptic curve. We now aim at computing the dimension of the moduli space $\M_{g}(\TT A,C)$. 

\begin{defi}
An irreducible parametrized tropical curve $h:\Gamma\to \TT A$ is said to be \textit{simple} if:
\begin{itemize}[label=-]
\item $\Gamma$ is trivalent,
\item $h$ is an immersion.
\end{itemize}
\end{defi}

In particular, a simple tropical curve no contracted edges, and no flat vertices (vertices whose adjacent edges belong to a line). We consider simple parametrized tropical curves.

\begin{lem}\label{lemma deformation of a cycle}
If $h:\Gamma\to \TT A$ is a simple parametrized tropical curve that has no contracted edges, and $\gamma$ is a cycle in $\Gamma$, it is always possible to deform the map along $\gamma$.
\end{lem}

\begin{proof}
Let $\gamma$ be a cycle in $\Gamma$, formed by the edges $e_1,\dots, e_n$, whose respective slopes are $u_1,\dots,u_n$. Let $v_i$ be the slope of the remaining edge coming at the vertex $V_i$ between $e_i$ and $e_{i+1}$, so that
$$u_{i+1}=u_i+v_i.$$
Then, move the vertex in the direction $v_i$:
$$V_i(t)=V_i+\frac{t}{\det(u_i,u_{i+1})}v_i.$$
As
\begin{align*}
\det(u_{i+1},V_{i+1}(t)-V_i(t)) & = \det(u_{i+1},V_{i+1}-V_i) + t\det\left(u_{i+1},\frac{v_{i+1}}{\det(u_i,u_{i+1})}-\frac{v_i}{\det(u_i,u_{i+1})}\right)\\
& = 0 + t\frac{\det(u_{i+1},v_{i+1})}{\det(u_{i+1},u_{i+2})}- t\frac{\det(u_{i+1},v_{i})}{\det(u_{i},u_{i+1})}\\
& =t-t\\
& = 0,\\
\end{align*}
the slope of the edges stays the same and we get a deformation of the tropical curve.
\end{proof}

The deformation of a cycle is illustrated on Figure \ref{figure deformation cycle} in the case of a contractible cycle and a non-contractible one.

\begin{figure}
\begin{center}
\begin{tabular}{cc}
\begin{tikzpicture}[line cap=round,line join=round,>=triangle 45,x=0.6cm,y=0.6cm]
\clip(0,0) rectangle (6,6);
\draw [line width=1.5pt] (1,1)-- ++(2,0)-- ++(2,2)-- ++(0,2)-- ++(-2,0)-- ++(-2,-2)-- ++(0,-2);
\draw [line width=1pt,dotted] (2,2)-- ++(1,0)-- ++(1,1)-- ++(0,1)-- ++(-1,0)-- ++(-1,-1)-- ++(0,-1);
\draw [line width=1.5pt] (0,0)-- ++(1,1);
\draw [line width=1.5pt] (5,5)-- ++(1,1);
\draw [line width=1.5pt] (3,0)-- ++(0,1);
\draw [line width=1.5pt] (3,5)-- ++(0,1);
\draw [line width=1.5pt] (0,3)-- ++(1,0);
\draw [line width=1.5pt] (5,3)-- ++(1,0);

\draw [line width=1pt,dotted] (1,1)-- ++(1,1);
\draw [line width=1pt,dotted] (4,4)-- ++(1,1);
\draw [line width=1pt,dotted] (3,1)-- ++(0,1);
\draw [line width=1pt,dotted] (3,4)-- ++(0,1);
\draw [line width=1pt,dotted] (1,3)-- ++(1,0);
\draw [line width=1pt,dotted] (4,3)-- ++(1,0);
\end{tikzpicture}
&
\begin{tikzpicture}[line cap=round,line join=round,>=triangle 45,x=0.6cm,y=0.6cm]
\clip(0,0) rectangle (7,5);

\draw [line width=1.5pt] (0,1)-- ++(2,2)-- ++(3,0)-- ++(2,2);
\draw [line width=1.5pt] (2,3)-- ++(0,2);
\draw [line width=1.5pt] (5,0)-- ++(0,3);

\draw [line width=1pt,dotted] (0,0)-- ++(2,2)-- ++(3,0)-- ++(2,2);
\draw [line width=1pt,dotted] (2,2)-- ++(0,1);
\end{tikzpicture} \\
$(a)$ & $(b)$ \\
\end{tabular}
\caption{\label{figure deformation cycle}Deformation of a contractible cycle and a non-contractible one}
\end{center}
\end{figure}
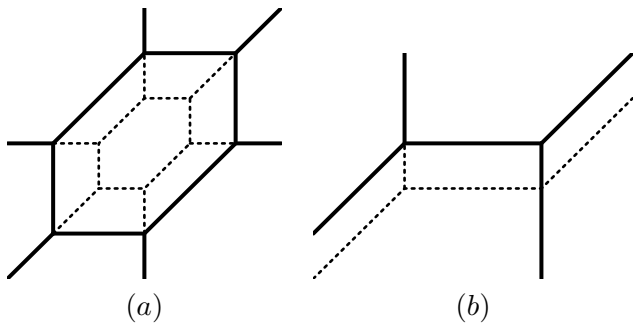

\begin{prop}\label{proposition dimension count}
We have the following facts:
\begin{enumerate}[label=(\roman*)]
\item If $\mathcal{M}_{g,C}^\Gamma$ is a combinatorial type with $\Gamma$ simple, the space of its deformations is of dimension $g$.
\item Let $h:\Gamma\to \TT A$ be a parametrized tropical curve with $h$ an immersion but $\Gamma$ not necessarily trivalent. Then $\Gamma$ can be deformed into a nearby tropical curve that is trivalent.
\item With assumptions as in $(ii)$, the space of its deformation is of dimension at most $g-1$.
\end{enumerate}
\end{prop}

\begin{proof}
\begin{enumerate}[label=(\roman*)]
\item Let $h:\Gamma\to \TT A$ be a simple tropical. Let $\Q$ be a lifting set for $\Gamma$, and assume that the curve $\widetilde{\Gamma}$ is rational. The set $\Q$ is of size $g$, the curve $\widetilde{\Gamma}$ has genus $0$ and $2g$ unbounded ends. As it is trivalent, it has a space of deformation of dimension $2g-1$.

By assumption, there are no contracted edge. Then, the only relation between the moment of the unbounded edges is provided by the tropical Menelaus theorem \cite{mikhalkin2017quantum}:
$$\sum_{i=1}^{2g}\mu_i=0,$$
where the scalar $\mu_i$ denote the moments of the unbounded ends of $\widetilde{\Gamma}$, indexed from $1$ to $2g$. Indeed, for any pair of unbounded ends, it is possible to move the chord between them, modifying only their moments, proving that the functions $\mu_i$ span a space of dimension exactly $2g-1$, since there are $2g$ unbounded ends.

Thus, the $g$ conditions in Proposition \ref{proposition deformation} impose in fact $g-1$ conditions, since the sum of them is $0$ by tropical Menelaus theorem. Then, the space of deformations of $\Gamma$ is of dimension $2g-1-(g-1)=g$.

\item To prove the second statement, assume that a curve $h:\Gamma\to \TT A$ is not trivalent, but $h$ is still an immersion. Let $V$ be a vertex which is not trivalent, and choose a loop in $\Gamma$ that contains the vertex $V$, and such that at any vertex of the loop, the two directions of the edges of the loop are not colinear. This is possible because $h$ is an immersion. Then, move a little bit the loop in one direction, keeping fixed all edges incident to the loop. This modifies the combinatorial type at each non-trivalent vertex. By induction, we can assume that there are no non-trivalent vertices left. As in \cite{blomme2021floor}, the only case where this is not possible is when $V$ is a quadrivalent vertex with a loop on it: an edge whose ends coincide but has non-zero slope. Such an edge would realize a circle, and this cannot happen since $\TT A$ does not contain any elliptic curve.

\item Assume that there exists a family of deformations of $h:\Gamma\to A$ of dimension at least $g$. Combining this family with a deformation of $\Gamma$ into a trivalent curve, we would obtain a trivalent curve that varies in a space of dimension greater than $g+1$. This is impossible by the first point. Hence, $\Gamma$ cannot vary in a space of dimension greater than $g$.
\end{enumerate}
\end{proof}

\begin{rem}\label{remark superabundancy}
Proposition \ref{proposition dimension count} proves that all parametrized tropical curves in $\TT A$ are superabundant, as the expected dimension is $g-1$: let $C\in H_{1,1}(\TT A)$ be a class that is realizable, and $g$ a fixed genus. The ``expected" dimension of the moduli space $\M_{g}(\TT A,C)$ of genus $g$ curves in the class $C$ is
$$3g-3 +2 -2g=g-1.$$
The $3g-3$ is the dimension of the space of curves of genus $g$, the $2$ accounts for the possible translations, and the $-2g$ for the conditions that the $g$ cycles impose on the curve. Fortunately, in this case, the superabundancy makes the dimension count match the complex dimension, which is also $g$. 
\end{rem}

If $h$ is not assumed to be an immersion, there can be ``supersuperabundant" curves if there are some contracted edges, flat vertices or cycles mapped to a line. Such curves vary in a space of too big dimension, but do not contribute any solutions to the enumerative problem that we consider because their image can be reparametrized by a parametrized tropical curve $h':\Gamma'\to \TT A$, where $h'$ is an immersion, and $\Gamma'$ is not trivalent, or has smaller genus.

\begin{rem}
If $g=\deg C+1$, the dimension of the deformation space is easy to compute since curves are generic sections of a line bundle on $\TT A$. Thus, deformations of a curve are corner locus of another section of the same line bundle, hence a dimension $\deg C-1=g-2$, or a translate of it, hence a ``$+2$", giving the expected $g$.
\end{rem}

\begin{rem}
In a sense, computing the dimension of the moduli space amounts to computing the dimension of the moduli space of planar tropical curves which are subject to some boundary conditions assuring that the gluing pass to the quotient. The superabundancy comes from the fact that the position of the boundary points is constrained by the tropical Menelaus theorem, as used in the proof of Proposition \ref{proposition dimension count}.
\end{rem}

\section{Enumerative problem and counts}

\subsection{Enumerative problem}

According to the previous section, it makes sense to count the number of genus $g$ irreducible tropical curves in the class $C$ that pass through $g$ points. If the points are chosen in generic position, Proposition \ref{proposition dimension count} ensures that the curves are simple. 

\medskip

Contrarily to the planar case, as the dimension depends linearly on the genus, the space of reducible curves with components of genera $g_1,\dots,g_r$ adding up to $g$ has also dimension $g$. We have thus two possible conventions regarding reducible curves.
\begin{itemize}[label=-]
\item The genus of a reducible curve $\Gamma$ can be defined by $1-\chi(\Gamma)$, where $\chi$ is the Euler characteristic. Concretely, it is equal to the sum of the genera of the components, plus one minus the number of components. This way, given a configuration $\P$ of $g$ points, we always have deformable genus $g$ curves that pass through $\P$, but these are reducible: the family of reducible curves $\Gamma_1\cup\cdots\cup\Gamma_r$ varies in dimension $g_1+\cdots+g_r=g+r$.
\item The other possibility is to set the genus of a reducible curve equal to the sum of the genera of the components, so that a curve of genus $g$ always varies in dimension $g$. However, we always have to account for reducible curves passing through a configuration $\P$ of $g$ points, although they do not vary in positive dimensional families anymore.
\end{itemize}
In the following, we avoid the problem of reducible curves by considering irreducible curves whenever possible.

\begin{expl}
On Figure \ref{figure dimension deformation space} $(b)$ we see a curve in the class $C=\begin{pmatrix}
2 & 0 \\
0 & 2 \\
\end{pmatrix}$ which is the union of two reducible components of genus $2$, and on $(a)$ an irreducible curve of genus $4$. Both curves vary in spaces of dimension $4$. This is especially easy for the reducible curves since the only possibility of deformation is to translate each of the reducible components.
\end{expl}

\begin{figure}
\begin{center}
\begin{tabular}{cc}
\begin{tikzpicture}[line cap=round,line join=round,>=triangle 45,x=0.4cm,y=0.4cm]
\clip(0,0) rectangle (10,10);
\draw [line width=1pt] (0,0)-- ++(8,2)-- ++(2,8)-- ++(-8,-2)-- ++(-2,-8);

\draw [line width=1.5pt] (0.5,2)-- ++(0.5,0)-- ++(1,1)-- ++(3,0)-- ++(1,1)-- ++(2.5,0);
\draw [line width=1.5pt] (1.25,5)-- ++(0.75,0)-- ++(2,2)-- ++(5.25,0);
\draw [line width=1.5pt] (1,0.25)-- ++(0,1.75);
\draw [line width=1.5pt] (5,1.25)-- ++(0,1.75);
\draw [line width=1.5pt] (2,3)-- ++(0,2);
\draw [line width=1.5pt] (6,4)-- ++(0,5);
\draw [line width=1.5pt] (4,7)-- ++(0,1.5);
\end{tikzpicture}
&
\begin{tikzpicture}[line cap=round,line join=round,>=triangle 45,x=0.4cm,y=0.4cm]
\clip(0,0) rectangle (10,10);
\draw [line width=1pt] (0,0)-- ++(8,2)-- ++(2,8)-- ++(-8,-2)-- ++(-2,-8);

\draw [line width=1.5pt] (1.25,5)-- ++(1.75,0)-- ++(2,2)-- ++(4.25,0);
\draw [line width=1.5pt] (0.75,3)-- ++(5.25,0)-- ++(2,2)-- ++(0.75,0);
\draw [line width=1.5pt] (3,0.75)-- ++(0,4.25);
\draw [line width=1.5pt] (6,1.5)-- ++(0,1.5);
\draw [line width=1.5pt] (5,7)-- ++(0,1.75);
\draw [line width=1.5pt] (8,5)-- ++(0,4.5);
\end{tikzpicture} \\
$(a)$ & $(b)$ \\
\end{tabular}
\caption{\label{figure dimension deformation space}On $(a)$ an irreducible curve of genus $4$ and on $(b)$ a reducible curve with two components of genus $2$.}
\end{center}
\end{figure}
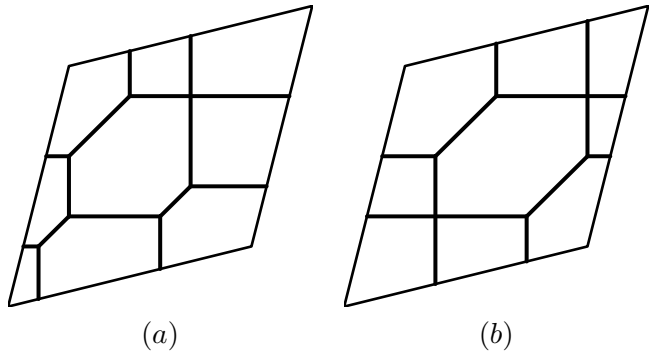

It is possible to partially get rid of reducible curves as follows. If $\TT A$ is chosen generically among the surfaces with curves in the class $C$, \textit{i.e.} the matrix of $S$ is generic among the real matrices such that $CS^T\in\S_2(\RR)$, then $C$ and its multiple are the only classes realizable by curves. If $S$ is not chosen generically, we have classes that may appear or disappear when we move $\TT A$ among the space of abelian surfaces with curves in the class $C$.

\medskip

We already know several results concerning enumeration of curves inside abelian surfaces:
\begin{itemize}[label=-]
\item In \cite{bryan1999generating}, without tropical geometry, J. Bryan and N. Leung give generating functions for the number of genus $g$ curves in a primitive class $C$ passing through $g$ points, or in a fixed linear system and passing through $g-2$ points. This result does not depend on the choice of the abelian surface as long as the class $C$ is realizable by some curve. In fact, an abelian surface is Hyperk\"ahler, \textit{i.e.} given a metric, it has a whole sphere of possible complex structures. For a class $C$, exactly one of the possible complex structures admits curves in the class $C$, and the number of such curves passing through $g$ points does not depend on the chosen abelian surface. This count includes reducible curves. They managed to compute these numbers for primitive classes $C$ by making an explicit computation for abelian surfaces which are product of two generic elliptic curves.

\item In \cite{halle2017tropical}, L. Halle and F. Rose use tropical geometry and maps between tropical tori to count the number of genus $2$ curves in a fixed class passing through $2$ points. They recover the result of \cite{bryan1999generating} in the genus $2$ case.

\item Finally, in \cite{nishinou2020realization}, T. Nishinou gives a correspondence theorem that allows one to compute the number of genus $g$ curves passing through $g$ points and that tropicalize to some given trivalent curve. This recovers the result from \cite{bryan1999generating} provided one is able to solve the analog tropical problem, at least for some abelian surface containing curves in the class $C$.
\end{itemize}

Let $\mathcal{P}\subset \TT A$ be a configuration of $g$ points chosen generically inside $\TT A$. Let $h:\Gamma\to \TT A$ be a tropical curve of genus $g$ in the class $C$ passing through $\mathcal{P}$. The tropical curve is \textit{rigid}, \textit{i.e.} it is not possible to deform it while still passing through the points of $\mathcal{P}$. We have the following.

\begin{prop}\label{proposition complement tree}
Let $h:\Gamma\to \TT A$ be an irreducible tropical curve of genus $g$ in the class $C$ passing through $\P$. Then:
\begin{itemize}[label=-]
\item the tropical curve is simple and rigid,
\item the complement $\Gamma\backslash h^{-1}(\mathcal{P})$ is a tree: it is connected and has no cycles.
\end{itemize}
\end{prop}

\begin{proof}
If a curve $h:\Gamma\to \TT A$ is not simple, then either $h$ is not an immersion, or $\Gamma$ is not trivalent. If $h$ is not an immersion, merging the edges that have a common vertex and the same image, it is possible to reparametrize the curve by an immersion with a new graph of smaller genus or vertices of higher valence. In either case, Proposition \ref{proposition dimension count} tells us that the curve varies in a space of dimension strictly smaller than $g$. Thus, a curve passing through a generic configuration $\P$ has to be simple. If the curve was not rigid, it would be possible to deform the curve $\Gamma$ in a $1$-parameter family, and get a curve that is not simple, contradicting the genericity of $\P$. Hence, the curve is rigid.

\medskip

The complement $\Gamma\backslash h^{-1}(\mathcal{P})$ has no cycles since it is always possible to deform a cycle using lemma \ref{lemma deformation of a cycle}. Knowing the complement of the marked points has no cycles, and that its Euler characteristic is $1-g+g=1$, we deduce that it is connected.
\end{proof}

In particular, the set $h^{-1}(\mathcal{P})$ is lifting for the curve $\Gamma$, and the complement has no cycles. Then, we deduce that each marked point $p$ defines a unique path $\lambda_p$ on $\Gamma\backslash h^{-1}(\mathcal{P})$: the one that connects the two halfs of the edge that it splits. In other words, it is the unique loop of $(\Gamma\backslash h^{-1}(\mathcal{P}))\cup\{p\}$. Moreover, this loop realizes some homology class inside $\TT A$, which can be either $0$ or nonzero, and that we also denote by $\lambda_p$.

\medskip

Using Proposition \ref{proposition complement tree} and the lifting procedure from the section \ref{section lifting}, it is now possible to \textit{unfold} the complement of the marked points, lifting it to the universal cover $N_\RR$ of $\TT A$, and get a planar parametrized tropical curve:
$$\tilde{h}:\widetilde{\Gamma}\longrightarrow N_\RR.$$
If the loop $\lambda_p$ associated to a point $p$ is non-zero, there are two distinct points in $N_\RR$ and two corresponding unbounded ends of $\widetilde{\Gamma}$. Otherwise, the point lifts to a single point and $\lambda_p$ to a loop inside $N_\RR$.

\medskip

We now focus on enumerative problems involving tropical curves in the torus. One important feature in which it differs from the planar setting, is that as noticed in \ref{remark superabundancy}, all curves are superabundant: the dimension of their moduli space is greater than the expected dimension. Fix a genus $g$ and a realizable class $C$. Let $\delta(C)$ be the integer length of the class, \textit{i.e.} the biggest integer by which $C$ is divisible. Notice that if a curve $\Gamma$ realizes the class $C$, its gcd (equal to the gcd of the weights of its edges) divides $\delta(C)$. We then consider the two following enumerative problems.
\begin{itemize}[label=-]
\item How many (irreducible) parametrized tropical curves of genus $g$ in the class $C$ pass through a generic configuration $\P$ of $g$ points ?
\item For $k|\delta(C)$, how many irreducible parametrized tropical curves of genus $g$ and gcd $k$ in the class $C$ pass through a generic configuration $\P$ of $g$ points ?
\end{itemize}
Counting reducible curves with a fixed gcd does not make sense because we only defined the gcd of an irreducible curve.

\medskip

As usual, the solutions to these enumerative problems need to be counted with some multiplicity in order to get an invariant, \textit{i.e.} a result that does not depend on the choice of $\P$ provided that it is generic.

\begin{defi}\label{definition multiplicity tropical curve}
Let $h:\Gamma\to\TT A$ be a simple tropical curve. For a vertex $V$ of $\Gamma$ with outgoing slopes $a_V$, $b_V$ and $-a_V-b_V$, let $m_V=|\det(a_V,b_V)|$ be its multiplicity. Then, we set:
\begin{itemize}[label=$\bullet$]
\item the usual multiplicity $m_\Gamma=\prod_{V\in V(\Gamma)}m_V\in \NN$,
\item the refined multiplicity $m^q_\Gamma=\prod_{V\in V(\Gamma)}[m_V]_q=\prod_{V\in V(\Gamma)}\frac{q^{m_V/2}-q^{-m_V/2}}{q^{1/2}-q^{-1/2}}\in\ZZ[q^{\pm 1/2}]$.
\end{itemize}
\end{defi}

The refined multiplicity is a symmetric Laurent polynomial in the variable $q$, which is the product of the quantum analogs of the vertex multiplicities. Let $h:\Gamma\to \TT A$ be a curve, and $\delta$ be an integer. The multiplication by $\delta$ multiplies the vertex multiplicities by $\delta^2$. As a simple genus $g$ curve has $2g-2$ vertices, we have
$$m_{\delta\Gamma}=\delta^{4g-4}m_\Gamma.$$
For the refined multiplicity, the relation is a little bit more complicated since the vertex multiplicities are involved in the exponent of the polynomials. The relation is
$$m_{\delta\Gamma}^q(q)=[\delta^2]_q^{2g-2}m_{\Gamma}^q(q^{\delta^2}).$$

Now, let $\P$ be a generic configuration of points on $\TT A$. We set the following:
\begin{align*}
N_{g,C,k}^{\trop}(\TT A,\P) = \sum_{\substack{h(\Gamma)\supset\P \\ \delta_\Gamma=k}} m_\Gamma , \\
BG_{g,C,k}(\TT A,\P) = \sum_{\substack{h(\Gamma)\supset\P \\ \delta_\Gamma=k}} m^q_\Gamma.\\
\end{align*}
The sums are over the irreducible curves of genus $g$ in the class $C$, having gcd $k$ for $k$ dividing the integral length $\delta(C)$ of $C$, that pass through the point configuration $\P$. Due to the relations between the multiplicities, and as the curves having gcd $k$ are obtained by multiplication of a curve with gcd $1$, we always have the following:
$$ N_{g,C,k}(\TT A,\P)=(k^2)^{2g-2} N_{g,C/k,1}(\TT A,\P) \text{ and } BG_{g,C,k}(\TT A,\P)= [k^2]_q^{2g-2} BG_{g,C/k,1}(\TT A,\P).$$

We can also consider the sums over all the irreducible curves passing through $\P$, without imposing the value of the gcd of the curves. However, we prove the invariance for the counts with the value of the gcd imposed, which is stronger than the invariance for the counts of all curves. In particular, we can choose multiplicities of the form $f(\delta_\Gamma)m_\Gamma$ and still get an invariant. We set
\begin{align*}
M_{g,C}(\TT A,\P) &  = \sum_{h(\Gamma)\supset\P} m_\Gamma = \sum_{k|\delta(C)} N_{g,C,k}(\TT A,\P) \in\NN , \\
N_{g,C}(\TT A,\P) &  = \sum_{h(\Gamma)\supset\P} \delta_\Gamma m_\Gamma = \sum_{k|\delta(C)}k N_{g,C,k}(\TT A,\P) \in\NN , \\
R_{g,C}(\TT A,\P) & = \sum_{h(\Gamma)\supset\P} \delta_\Gamma m^q_\Gamma \in \ZZ[q^{\pm 1/2}] , \\
BG_{g,C}(\TT A,\P) & = \sum_{h(\Gamma)\supset\P} m^q_\Gamma = \sum_{k|\delta(C)} BG_{g,C,k}(\TT A,\P) \in \ZZ[q^{\pm 1/2}] . \\
\end{align*}

These are four different sums over the solutions. The first is with multiplicity $m_\Gamma$ and the second with multiplicity $\delta_\Gamma m_\Gamma$. The third is a refinement of the second since it specializes to it when $q$ is set equal to $1$. The multiplicity has been replaced by $\delta_\Gamma m_\Gamma^q$. The last is also a refined count of solutions, but of the first. The multiplicity does not involve the gcd of the curves, only the vertex multiplicities.

\begin{rem}
Perhaps some other choices of multiplicities would yield invariant counts that are related to invariant counts on complex or real abelian surfaces, similarly to the relation between refined invariants in toric varieties and refined counts of real curves, as investigated in \cite{mikhalkin2017quantum} and \cite{blomme2021refinedreal}.
\end{rem}

\begin{expl}
We refer to the end of the paper for examples and computations, once the pearl diagram algorithm has been described. For now, we just give the contribution of the curve depicted on Figure \ref{figure unique genus 2 curve through 2 points} to the various counts:
\begin{align*}
M_{2,2I_2}^{\trop} & \to  4\times 4=16,\\
N_{2,2I_2}^{\trop} & \to 2\times 4\times 4=32,\\
R_{2,2I_2} & \to 2\times (q^{1/2}+q^{-1/2})\times (q^{1/2}+q^{-1/2})=2(q+2+q^{-1}),\\
BG_{2,2I_2} & \to q+2+q^{-1}.\\
\end{align*}
\end{expl}

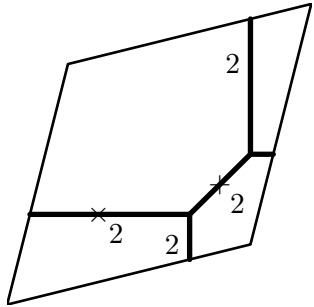
\begin{figure}
\begin{center}
\begin{tikzpicture}[line cap=round,line join=round,>=triangle 45,x=0.4cm,y=0.4cm]
\clip(0,0) rectangle (10,10);
\draw [line width=1pt] (0,0)-- ++(8,2)-- ++(2,8)-- ++(-8,-2)-- ++(-2,-8);

\draw [line width=2pt] (0.75,3)-- ++(5.25,0)-- ++(2,2)-- ++(0.75,0);
\draw [line width=2pt] (6,1.5)-- ++(0,1.5);
\draw [line width=2pt] (8,5)-- ++(0,4.5);
\draw (3,3) node {$\times$};
\draw (7,4) node {$+$};
\draw (3,3) node[below right] {$2$};
\draw (7,4) node[below right] {$2$};
\draw (6,2) node[left] {$2$};
\draw (8,8) node[left] {$2$};

\end{tikzpicture}
\caption{\label{figure unique genus 2 curve through 2 points}The unique genus $2$ curve passing through two points.}
\end{center}
\end{figure}

\subsection{Correspondence theorem}
\label{section correspondence theorem}

In \cite{nishinou2020realization}, T. Nishinou gives a realization theorem and a correspondence theorem for genus $g$ irreducible curves passing through $g$ points in general position on the torus that we now recall. More precisely, he considers the following family of complex algebraic tori:
$$\CC A_t=N_{\CC^*}/\Lambda_t,$$
as in section \ref{section degeneration}. With the same notations, $\Lambda_t$ is generated by $\alpha_1 t^{s_1},\alpha_2 t^{s_2}\in N_{\CC^*}$. The family tropicalizes to the tropical torus
$$\TT A=N_\RR/\Lambda,$$
with $\Lambda=\mathrm{Im} S$. Let $C\in \Lambda\otimes N$ be a class, that we assume to be primitive, otherwise we divide it by its integer length. Let $(e_1,e_2)$ be a basis of $N$, and $(e_1^*,e_2^*)$ be the dual basis of $M$. Assume the slope of $h$ on the edges that intersect $B_1$ is $a_ie_1+b_i e_2$, and $c_ie_1+d_ie_2$ for $B_2$. The vectors associated through Poincar\'e duality are $-b_i e_1^*+a_i e_2^*$ and $-d_i e_1^*+c_ie_2^*$. Assume that $\Lambda_t\subset N_{\CC^*}$ is generated by $(\alpha_{11} t^{s_{11}},\alpha_{12}t^{s_{12}})$ and $(\alpha_{21} t^{s_{21}},\alpha_{22}t^{s_{22}})$. Then, we consider the following quantity, introduced in \cite{nishinou2020realization}:
$$\sigma(C)=\prod_i \alpha_{11}^{-b_i}\alpha_{12}^{a_i}\prod_j \alpha_{21}^{-d_i}\alpha_{22}^{c_i}.$$
For a version free of coordinates choices, for $t=1$, we have $j:\Lambda_1\to N_{\CC^*}$. Thus, we also have the following map, also denoted by $\sigma$:
$$\begin{array}{rccl}
\sigma: & \Lambda\otimes N\simeq\Lambda_1\otimes N & \longrightarrow & \CC^* \\
 & \lambda_1\otimes n & \longmapsto & \mathrm{PD}(n)\cdot j(\lambda_1) \\
\end{array},$$
where $\mathrm{PD}(n)$ is the vector Poincar\'e dual to $n$, $j(\lambda_1)\in N_{\CC^*}$, and the action of $M$ on $N_{\CC^*}$ is the evaluation of the monomial: $m\cdot z^n=z^{\langle m,n\rangle}\in\CC^*$. The complex number $\sigma(C)$ is then the value of $\sigma$ on the class $C$. We now have the following realization theorem.

\begin{theom}(Nishinou \cite{nishinou2020realization})
Let $h:\Gamma\to \TT A$ be a simple tropical curve realizing the class $\delta C$, and let $\delta_\Gamma$ be the gcd of $\Gamma$. The curve is realizable if and only if
$$\sigma\left(\frac{\delta C}{\delta_\Gamma}\right)=(-1)^{\sum m_V/\delta_\Gamma}.$$
\end{theom}

The relation from Nishinou's realization's theorem is obtained by making the product of Menelaus relations at each vertex of the tropical curve, and is equivalent to the curve being ``phase-realizable". Here, Menelaus relation means the relation (46) in Proposition 39 from \cite{mikhalkin2017quantum}. It is the relation between the phases of the edges adjacent to a vertex for a tropical curve that is realizable.

\medskip

The necessary condition from Nishinou's theorem corresponds in fact to the necessary condition that the class $C$ of the curve satisfies Riemann bilinear relations for members of the family $\CC A_t$. In the condition given in section \ref{section degeneration}, we do not have any condition on the tropical curve, \textit{i.e.} the right-hand side does not appear. This is due to the following lemma, that asserts that the right hand-side of the equality in Nishinou's theorem is always $1$.

\begin{lem}
Let $h:\Gamma\to\TT A$ be a simple parametrized tropical curve, and let $\delta_\Gamma$ be the gcd of the curve. Then
$$\sum_{V\in V(\Gamma)}\frac{m_V}{\delta_\Gamma}\equiv 0 \mod 2.$$
\end{lem}

\begin{proof}
First, assume that $\delta_\Gamma$ is even. Then, as for any vertex the multiplicity is divisible by $\delta_\Gamma^2$, each term in the sum is still even, and $\sum_{V\in V(\Gamma)}\frac{m_V}{\delta_\Gamma}$ is also even.

Assume that $\delta_\Gamma$ is odd. Then, the sum has the same parity as $\sum_{V\in V(\Gamma)}m_V$. This sum is equal mod $2$ to the intersection number $C\cdot C$. To see this, consider the curve $\Gamma$ and a small translate of it. The intersection points come from the vertices and double points:
\begin{itemize}[label=-]
\item each double point contributes two intersection points with the same multiplicity,
\item each trivalent vertex contributes one intersection points of multiplicity $m_V$.
\end{itemize}
As the intersection form on $H_{1,1}(\TT A)$ is even, since it is given in the basis $(\lambda_i\otimes e_j)$ by the matrix
$$\begin{pmatrix}
0 & 1 & 0 & 0 \\
1 & 0 & 0 & 0 \\
0 & 0 & 0 & 1 \\
0 & 0 & 1 & 0 \\
\end{pmatrix},$$
the result follows.
\end{proof}

Then, choosing $g$ families of points $\mathcal{P}_t$ inside $\CC A_t$ that tropicalize to a set $\mathcal{P}$ of $g$ points inside $\TT A$ which are in generic position, Nishinou gives a formula to count the number of irreducible curves inside $\CC A_t$ passing through $\mathcal{P}_t$ that tropicalize to a given simple tropical curve passing through $\mathcal{P}$. The multiplicity $m_\Gamma^\CC$ of a tropical curve given by the correspondence theorem from \cite{nishinou2020realization}. Let $h:\Gamma\to \TT A$ be a trivalent tropical curve passing through $\mathcal{P}$, and let $\Gamma'$ be the graph $\Gamma$ subdivided by the points of $h^{-1}(\mathcal{P})$. First, choose an orientation of every edge $e\in E(\Gamma')$. Then, we have a map
$$\begin{array}{rccl}
\Theta:& \bigoplus_{v\in V(\Gamma')}N  & \longrightarrow & \bigoplus_{e\in E(\Gamma')}N/N_e \oplus\bigoplus_{1}^g N\\
 & (\phi_v) & \longmapsto & \left( \phi_e,\phi_{v_i}\right) \\
 \end{array},$$
where $N/N_e$ is the $1$-dimensional lattice obtained by quotienting $N$ by the primitive slope of the edge $e$, $\phi_e=\phi_{\partial^+ e}-\phi_{\partial^- e}\in N/N_e$ is the difference between the extremities if the oriented edge $e$ projected to the quotient, and $v_i$ is the vertex of $\Gamma'$ associated to a marked point $p_i\in\mathcal{P}$. The dimension of the domain is $2|V(\Gamma')|=2|V(\Gamma)|+2g$ while the dimension of the codomain is $|E(\Gamma')|+2g=|E(\Gamma)|+3g$. As $\Gamma$ is of genus $g$ and is trivalent, we know that
$$\left\{ \begin{array}{l}
3|V(\Gamma)|=2|E(\Gamma)| ,\\
|V(\Gamma)|-|E(\Gamma)|=1-g.\\
\end{array}\right.$$
Hence,
$$(|E(\Gamma)|+3g)-(2|V(\Gamma)|+2g)=|V(\Gamma)|-|E(\Gamma)|+g=1,$$
and contrarily to the usual planar setting, domain and codomain do not have the same dimension. In fact, this map is not surjective precisely because of the superabundancy of the tropical curves. The multiplicity is defined as follows.

\begin{defi}\label{definition multiplicity nishinou}
The multiplicity of a trivalent tropical curve $h:\Gamma\to A$ passing through $\mathcal{P}$ is 
$$m^\CC_\Gamma=|\ker\Theta_{\CC^*}|\prod_{e\in E(\Gamma')} w_e,$$
where $\Theta_{\CC^*}$ is the map $\Theta\otimes\CC^*$.
\end{defi}

\begin{rem}
In the usual correspondence theorems, such as in \cite{nishinou2006toric} and \cite{tyomkin2017enumeration}, we have a similar map between lattices. This map is used to compute the number of possible \textit{phases} that one can put on the tropical curve. A \textit{phase} is a complex number such that at each vertex, the Menelaus condition is satisfied. See \cite{mikhalkin2005enumerative}, \cite{nishinou2006toric} for more details. The difference is that in those last cases, domain and codomain have the same dimension, and the multiplicity is then equal to the determinant of the map $\Theta$ between lattices. Here, we cannot take the determinant as the two spaces do not have the same dimension. Nevertheless, $|\ker\Theta_{\CC^*}|$ is also equal to the cardinal of the torsion part of $\mathrm{coker}\Theta$, as we see by applying the tensor product to the following small exact sequence:
$$0\to \ZZ^{2 |V(\Gamma)|+2g} \xrightarrow{\Theta} \ZZ^{|E(\Gamma)|+3g} \to \ZZ \oplus G \to 0,$$
with $G$ the torsion part of the cokernel. Applying the functor $-\otimes\CC^*$, we get
$$0\to \mathrm{Tor}(G,\CC^*) \to (\CC^*)^{2 |V(\Gamma)|+2g} \xrightarrow{\Theta_{\CC^*}} (\CC^*)^{|E(\Gamma)|+3g} \to \CC^* \to 0,$$
and as $\CC^*$ is divisible, we have $\mathrm{Tor}(G,\CC^*)\simeq G$. The cokernel can then be computed using exterior powers of the map between the lattices.
\end{rem}

We now have the following correspondence theorem.

\begin{theom}(Nishinou \cite{nishinou2020realization})
Let $h:\Gamma\to \TT A$ be a simple tropical curve passing through $\mathcal{P}$. The number of irreducible parametrized curves passing through $\mathcal{P}_t$ that tropicalize to $(h,\Gamma)$ is
$$m_\Gamma^\CC=|\ker \Theta_{\CC^*}|\prod_{e\in E(\Gamma')} w_e.$$
In particular, if $\N_{g,C}$ denotes the number of complex genus $g$ curves in the class $C$ passing through a generic configuration of $g$ points inside an abelian surface, then
$$N_{g,C}=\N_{g,C}.$$
\end{theom}

\subsection{Statement of the results}

Using Proposition \ref{proposition complement tree} to relate the tropical curve inside $\TT A$ to a planar tropical curve inside $N_\RR$, we can compute the multiplicity and get a formula analog to the multiplicity of a planar tropical curve provided by the correspondence theorem of Mikhalkin \cite{mikhalkin2005enumerative}.

\begin{theo}\label{theorem multiplicity product}
The multiplicity $m_\Gamma^\CC$ provided by Nishinou's correspondence theorem splits as follows:
$$m_\Gamma^\CC=\delta_\Gamma\prod_{V\in V(\Gamma)}m_V = \delta_\Gamma m_\Gamma,$$
where $m_\Gamma$ is the first multiplicity from Definition \ref{definition multiplicity tropical curve}, and $\delta_\Gamma$ is the gcd of the tropical curve $\Gamma$.
\end{theo}

\begin{rem}\label{remark mult by delta}
The presence of the factor $\delta$ is natural: assume the formula has been proven when $\delta$ is $1$. To get a curve with gcd $\delta>1$, we multiply all the weights of the edges by $\Delta$. As there are $3g-3+g=4g-3$ bounded edges on $\Gamma'$, the formula from Definition \ref{definition multiplicity nishinou} is multiplied by $\delta^{4g-3}$ since $\Theta$ is unchanged, only the weights of the bounded edges changes. Meanwhile, as $m_V$ becomes $\delta^2 m_V$ when the weights are multiplied by $\delta$, the product $\prod_V m_V$ is multiplied by $\delta^{2 V(\Gamma)}=\delta^{4g-4}$. Thus, we need an additional $\delta$ to have the same exponent of homogenization.
\end{rem}

\begin{expl}
\begin{itemize}[label=-]
\item Let us consider the genus $2$ curve passing through $2$ points on Figure \ref{figure unique genus 2 curve through 2 points}. Its gcd is $2$ and both vertices have multiplicity $4$. The Theorem asserts that its complex multiplicity is $2\times 4\times 4=32$. Alternatively, the subdivided curve has $5$ edges of weight $2$. We can see that there is a unique way of putting phases of the curves, which amounts to compute $|\ker\Theta_{\CC^*}|=1$. The multiplicity is still $32$.
\item Consider the two curves on Figure \ref{figure computation multiplicity examples different marking}. Both have multiplicity $4$, since the two vertices have multiplicity $2$, and its gcd is $1$. However, the formula $|\ker\Theta_{\CC^*}|$ does not yield the result in the same way. For the curve in $(a)$, the edge of weight $2$ is marked and counts double in the edges weights product. We can check that $|\ker\Theta_{\CC^*}|$ is $1$ since there is a unique way to put phases on the edges. For the curve on $(b)$, there are two ways to put a phase on the edge of weight $2$, meaning $|\ker\Theta_{\CC^*}|=2$, and it contributes only once in the product of weights, also leading to $4$. This is natural because the count of solutions cannot depend on the choice of the points, since it is equal to some invariant $N_{g,C}$ using the correspondence Theorem.
\end{itemize}
\end{expl}

\begin{figure}
\begin{center}
\begin{tabular}{cc}
\begin{tikzpicture}[line cap=round,line join=round,>=triangle 45,x=0.5cm,y=0.5cm]
\clip(0,0) rectangle (10,10);
\draw [line width=0.5pt] (2,4)-- (0,10);
\draw [line width=0.5pt] (2,4)-- (10,0);
\draw [line width=0.5pt] (0,10)-- (8,6);
\draw [line width=0.5pt] (10,0)-- (8,6);

\draw [line width=2pt] (1.33333333333,6)-- (4,6);
\draw [line width=1.5pt] (4,6)-- (8,2);
\draw [line width=1.5pt] (4,6)-- (5.333333333,7.33333333);
\draw [line width=1.5pt] (8,2)-- (7.333333333,1.33333333);
\draw [line width=2pt] (8,2)-- (9.3333333333,2);

\draw (2,6) node {$\times$};
\draw (6,4) node {$+$};

\draw (2,7) node {$2$};
\draw (8.5,2.5) node {$2$};

\end{tikzpicture}
&
\begin{tikzpicture}[line cap=round,line join=round,>=triangle 45,x=0.5cm,y=0.5cm]
\clip(0,0) rectangle (10,10);
\draw [line width=0.5pt] (2,4)-- (0,10);
\draw [line width=0.5pt] (2,4)-- (10,0);
\draw [line width=0.5pt] (0,10)-- (8,6);
\draw [line width=0.5pt] (10,0)-- (8,6);

\draw [line width=2pt] (1.33333333333,6)-- (4,6);
\draw [line width=1.5pt] (4,6)-- (8,2);
\draw [line width=1.5pt] (4,6)-- (5.333333333,7.33333333);
\draw [line width=1.5pt] (8,2)-- (7.333333333,1.33333333);
\draw [line width=2pt] (8,2)-- (9.3333333333,2);

\draw (5,7) node {$+$};
\draw (6,4) node {$+$};

\draw (2,7) node {$2$};
\draw (8.5,2.5) node {$2$};

\end{tikzpicture}
\\
$(a)$ & $(b)$ \\
\end{tabular}
\caption{\label{figure computation multiplicity examples different marking}Two tropical curves with different marking.}
\end{center}
\end{figure}
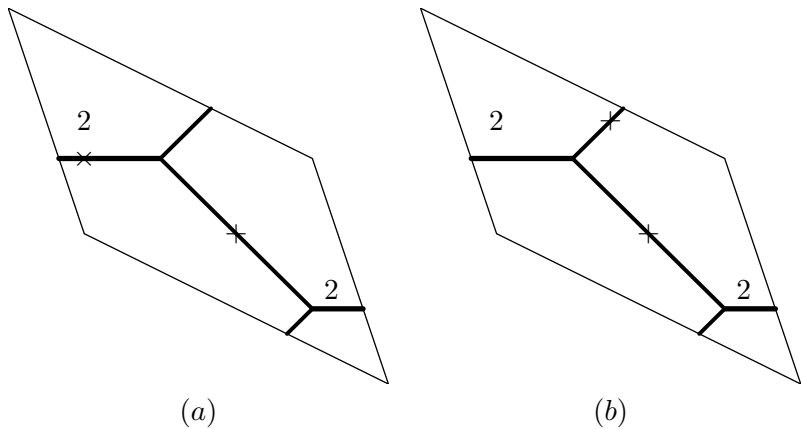

The gcd $\delta$ of the weights of the edges is not necessarily constant on the moduli space $\M_g(\TT A,C)$, nor for the solutions of the enumerative problems. However, it is constant on the connected components of the complement of some codimension $1$ substrata. And this suffices to provide invariance for the counts of solutions with a fixed gcd. The proof Theorem \ref{theorem multiplicity product} is postponed to section \ref{section proof multiplicity product}.

\medskip

Just as in the usual planar case, it is natural to try to replace the vertex multiplicity $m_V$ by its quantum analog $[m_V]_q=\frac{q^{m_V/2}-q^{-m_V/2}}{q^{1/2}-q^{-1/2}}$, in an attempt to obtain tropical refined invariants similar to Block-G\"ottsche refined invariants for toric surfaces. This justifies the definition of the refined tropical multiplicities from Definition \ref{definition multiplicity tropical curve}.

\medskip

Recall that we defined several refined tropical counts: $BG_{g,C,k}(\TT A,\P)$, $R_{g,C}(\TT A,\P)$ and $BG_{g,C}(\TT A,\P)$. We then have the following invariance statements:

\begin{theo}\label{theorem point invariance}
The refined count $BG_{g,C,k}(\TT A,\P)$ does not depend on the choice of $\mathcal{P}$ as long as it is generic.
\end{theo}

\begin{coro}
The counts $N_{g,C,k}^\trop(\TT A,\P)$, $R_{g,C}(\TT A,\P)$, $BG_{g,C}(\TT A,\P)$, $M_{g,C}(\TT A,\P)$ and $N_{g,C}(\TT A,\P)$ do not depend on the choice of $\mathcal{P}$ as long as it is generic.
\end{coro}

The proof is in section \ref{section invariance points}. The corollary follows from the expression of the respective enumerative counts in terms of $BG_{g,C,k}(\TT A,\P)$, which is invariant by Theorem \ref{theorem point invariance}, evaluating at $q=1$ for the integer valued counts. We remove the $\P$ from the notation to denote the corresponding invariants. We now partially get rid of the dependence in $\TT A$.

\begin{theo}\label{theorem torus invariant}
The refined invariant $BG_{g,C,k}(\TT A)$ does not depend on $\TT A$ as long as $\TT A$ contains curves in the class $C$ and is chosen generically among them.
\end{theo}

\begin{coro}
The invariants $N_{g,C,k}^\trop(\TT A)$, $R_{g,C}(\TT A)$, $BG_{g,C}(\TT A)$, $M_{g,C}(\TT A)$ and $N_{g,C}(\TT A)$ do not depend on the choice of $\TT A$ as long as $\TT A$ contains curves in the class $C$ and is chosen generically among them.
\end{coro}

The proof is in section \ref{section invariance surface}. Notice that if $\TT A$ is not chosen generically, we still have invariants for $\TT A$ since the proof of Theorem \ref{theorem point invariance} does not require a generic choice of $\TT A$.

\begin{rem}
If $\TT A$ is not chosen generically, the class $C$ might decompose as a sum $C_1+C_2$. Then, there are solutions which are reducible curves having two components in the classes $C_1$ and $C_2$. For instance, in the computation of the invariants of \cite{bryan1999generating}, the authors choose the abelian surface to be a product of two elliptic curves $S_1\times S_2$ and count curve in the class $[S_1] + n[S_2]$. All the curves in this class are reducible with all their components being genus $1$ curves. Tropically, this would mean that all the tropical curve would have refined multiplicity $1$ since there are no vertices. Curves in the classes $C_1$ and $C_2$ are not deformable along a deformation of $\TT A$, since the intersection points between the components would become a pair of trivalent vertices. The only classes that continue to exist are the one colinear with $C$.
\end{rem}

Due to the preceding remark, the refined invariant depends on the choice of the matrix $S$, and more precisely on cone of realizable classes inside the the group $\{C\in \Lambda\otimes N \text{ such that }C S^T\in\S_2(\RR)\}$.

\begin{expl}
Consider the class $C=\begin{pmatrix}
1 & 0 \\
0 & 1 \\
\end{pmatrix}$, which are curves as depicted on Figure \ref{figure example tropical curves} $(a)$. If the torus $\TT A$ is not chosen generically, the genus $2$ curve might decompose as a union of two tropical elliptic curves.There are no trivalent vertices anymore. Thus, we do not have any irreducible curve to count anymore, and the value of the invariant is now $0$, while it is $1$ for a generic choice of $\TT A$.
\end{expl}

\section{Proof of the results}

\subsection{Multiplicity as a product}
\label{section proof multiplicity product}

\begin{proof}[Proof of Theorem \ref{theorem multiplicity product}]
For each edge $e\in E(\Gamma')$ directed by some primitive vector $u_e$, we use $\det(u_e,-)$ as a coordinate on $N/N_e$. The map $\Theta$ starts from a lattice of rank $2|V(\Gamma')|=6g-4$ to a lattice of rank $|E(\Gamma')|+g=6g-3$. It is not surjective. In fact, its image always lie in the kernel of the following linear form defined on the codomain:
$$\varphi(\phi_e,p_i)=\sum_e \frac{w_e}{\delta}\phi_e,$$
where $\delta$ is the gcd of the curve. Indeed, for each $v\in V(\Gamma')$, we have
$$\varphi\circ\Theta(p_v)=\frac{1}{\delta}\left( w_1 \det(u_1,\varepsilon_1 p_v)+ w_2 \det(u_2,\varepsilon_2 p_v) + w_3 \det(u_3,\varepsilon_3 p_v)\right),$$
where $\varepsilon_i=\pm 1$ whether $p_v$ is the head or the tail of $e_i$. This is $0$ by the balancing condition. This does not happen in the planar case due to the presence of unbounded ends. As $\Theta$ is injective, taking a basis of the domain and codomain, we have a short exact sequence,
$$0\to \ZZ^{6g-4}\xrightarrow{\Theta} \ZZ^{6g-3}\to \ZZ\oplus G\to 0,$$
where $G$ is a torsion group. We want to compute the cardinal of this torsion group, because
$$0\to \mathrm{Tor}(G,\CC^*)\to (\CC^*)^{6g-4}\xrightarrow{\Theta_{\CC^*}} (\CC^*)^{6g-3}\to \CC^*\to 0.$$
In fact, consider the exterior power $\Lambda^{6g-4}\Theta:\Lambda^{6g-4}\ZZ^{6g-4}\simeq\ZZ\to \Lambda^{6g-4}\ZZ^{6g-3}\simeq\ZZ^{6g-3}$. It maps a generator of $\Lambda^{6g-4}\ZZ^{6g-4}$ to some multiple of the Pl\"ucker vector of the $\ker\varphi$, in which the image of the map lies. We are in fact interested in the absolute value of the proportionality constant. Let us rephrase it. In suitable basis, the matrix of $\Theta$ has the following form:
$$\begin{pmatrix}
d_1 & 0 & \cdots & 0 \\
0 & d_2 & & \vdots \\
\vdots & & \ddots & \vdots \\
0 & \cdots & \cdots & d_{6g-4} \\
0 & \cdots & \cdots & 0 \\
\end{pmatrix}.$$
The cokernel of $\Theta$ is $\ZZ\oplus G$ where $G=\bigoplus_1^{6g-4}\ZZ/d_i\ZZ$, and $|G|=\prod d_i$. The map $\varphi$ sends a point in $\ZZ^{6g-3}$ to its last coordinate. The Pl\"ucker vector of $\ker\varphi$ is $e_1\wedge\dots\wedge e_{6g-4}\in\Lambda^{6g-4}\ZZ^{6g-3}$, and the image of the generator of $\Lambda^{6g-4}\ZZ^{6g-4}$ is $\bigwedge (d_i e_i)=\left(\prod d_i\right)\bigwedge e_i$. We are interested in the constant $d_1\cdots d_{6g-4}$, which is both the cardinal of the torsion part $G$ of the cokernel, and the integral length of the image of a generator of $\Lambda^{6g-4}\ZZ^{6g-4}$ by $\Lambda^{6g-4}\Theta$. If the matrix is not given in bases that make it a diagonal matrix, the integer $\prod d_i$ is equal to the gcd of the maximal minors of the matrix.

\medskip

As there is a relation between the coordinates corresponding to $\bigoplus_{E(\Gamma')}N/N_e$, any maximal minor obtained by forgetting a row from the coordinates of $\bigoplus_1^g N$ is $0$. The only non-zero maximal minors come from forgetting one of the edges of $\Gamma'$. Let $e_0$ be such an edge. The matrix of $\Theta$ contains blocks of the following form:
$$\begin{array}{|c|}
\hline
\det(u_e,-) \\
\hline
I_2 \\
\hline
\det(u_e,-) \\
\hline
\end{array},$$
where columns are the coordinates of a marked point $p$, and rows are its position, and the coordinates corresponding to the two adjacent edges. The block $I_2$ is the only non-zero block on the row of the matrix of $\Theta$. In particular, we can use it to develop each maximal minors with respect to these rows. Then, as the complement of $\P$ inside the curve $\Gamma$ is a tree, we can assume that the edges have been oriented toward the edge whose row we are forgetting. We have already develop the minor with respect to the rows of the marked points. Now, let $V$ be a vertex adjacent to two leafs of the tree, with primitive directing vectors $u_1$ and $u_2$, and outgoing slope $u_e$. The balancing condition is $w_e u_e=w_1u_1+w_2u_2$. The corresponding block is
$$\begin{array}{|c|}
\hline
\det(u_e,-) \\
\hline
\det(u_1,-) \\
\hline
\det(u_2,-) \\
\hline
\end{array},$$
where columns are coordinates for the copy of $N$ corresponding to $V$, and the rows to the adjacent edges. As $V$ is adjacent to two leafs, the last two rows of the block are the only non-zero elements in the row of the minor. Thus, we can develop with respect to these two rows and get the smaller determinant where the vertex has been deleted, and the determinant multiplied by $\det(u_1,u_2)=\frac{m_V}{w_1w_2}$. In the end, we get $\frac{1}{\prod_{e\neq e_0} w_e}\prod_V m_V=\frac{w_{e_0}}{\prod_e w_e}\prod_V m_V$.

\medskip

Recall $\delta$ is the gcd of the weights of the edges. Assume that $\delta=1$. By Bezout's theorem, there exists some integer numbers $\alpha_e$ such that $\sum \alpha_e w_e=1$. Thus, we have
$$\sum \alpha_{e_0}\frac{w_{e_0}}{\prod_{e\in E(\Gamma')}w_e}\prod_V m_V=\frac{1}{\prod_{e\in E(\Gamma')}w_e}\prod_V m_V,$$
which is also an integer. The gcd of the $\frac{w_{e_0}}{\prod_{e\in E(\Gamma')}w_e}\prod_V m_V$ is then equal to $\frac{1}{\prod_{e\in E(\Gamma')}w_e}\prod_V m_V$. Multiplying by $\prod_{e\in E(\Gamma')}w_e$, we get the result. If $\delta>1$, we consider the tropical curve where all the weights have been divided by $\delta$. It has the same matrix $\Theta$ and the same minors. Let $w_e=\delta \widehat{w_e}$, where $\widehat{w_e}$ are the weights of the curve divided by $\delta$. For a vertex $V$, its multiplicity in the curve divided by $\delta$ is $\widehat{m_V}=\frac{m_V}{\delta^2}$. The value of the minors are
$$\frac{\widehat{w_{e_0}}}{\prod_{e\in E(\Gamma')}\widehat{w_e}}\prod_V \widehat{m_V}.$$
As in the case $\delta=1$, their gcd is 
$$\frac{1}{\prod_{e\in E(\Gamma')}\widehat{w_e}}\prod_V \widehat{m_V}=\delta\frac{1}{\prod_{e\in E(\Gamma')}w_e}\prod_V m_V,$$
as in Remark \ref{remark mult by delta}. Multiplying by the weight of the bounded edges, the result follows.
\end{proof}

\begin{rem}
Following \cite{nishinou2020realization}, the cardinal of $\mathrm{Coker}\Theta$ is equal to the number a possible phases of the tropical curve. The computation of the multiplicity done in the proof is more in the spirit of \cite{nishinou2006toric}. It should also be possible to compute to be $\delta\frac{\prod m_V}{\prod w_e}$ in a way similar to the computation of the multiplicity in the planar setting from \cite{mikhalkin2005enumerative}, by directly computing the possible number of possible phases on the edges.
\end{rem}

\subsection{Independence of the choice of the points}
\label{section invariance points}

Before getting to the proof of invariance, we notice the following. The invariance for the enumerative problem without imposing the gcd of the curves can be deduced from the invariance of the enumerative problems where the gcd is involved. Conversely, a curve of gcd $\delta$ in the class $C$ is obtained by multiplying by $\delta$ a primitive curve in the class $C/\delta$. Therefore, initializing with primitive classes $C$ for which all the curves realizing the class are also primitive, one can also recover the invariance of the counts with a fixed gcd from the invariance of the global enumerative problem.

\begin{proof}[Proof of Theorem \ref{theorem point invariance}]
Let $(\mathcal{P}_s)_{s\in [0;1]}$ be a generic path between two generic point configurations $\mathcal{P}_0$ and $\mathcal{P}_1$ inside the tropical torus $\TT A=N_\RR/\Lambda$. By composing different paths, we can always assume that only one of the points in the configuration moves. We intend to prove the invariance for the refined count $BG_{g,C,k}(\TT A,\P_s)$ when moving one point.

\medskip

If $\mathcal{P}_s$ is a generic configuration, and $\Gamma$ is a simple tropical curve with gcd $k$ passing through $\mathcal{P}_s$, considering as lifting set $\{p\in\P_s:\lambda_p\neq 0\}$, it is easy to see that a small deformation of the point configuration leads to a small deformation of the curve $\Gamma$ which stays in the same combinatorial type. In fact, only the cycle $\lambda_{p_s}$ moves, where $p_s$ is the moving point, as depicted on Figure \ref{figure point deformation cycle}. As the multiplicity is constant on the combinatorial type, we get that the refined count of solutions is locally invariant. We now have to check what happens when the point configuration crosses a \textit{wall}, where the configuration $\P_s$ is not generic anymore.

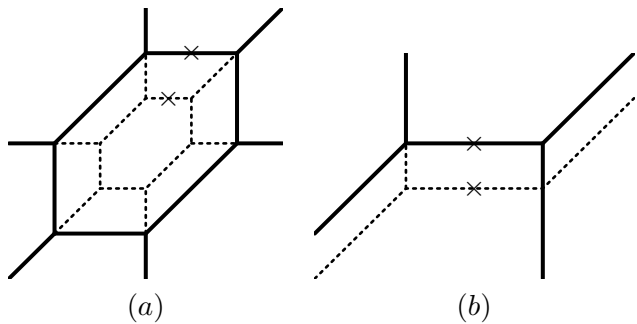
\begin{figure}
\begin{center}
\begin{tabular}{cc}
\begin{tikzpicture}[line cap=round,line join=round,>=triangle 45,x=0.6cm,y=0.6cm]
\clip(0,0) rectangle (6,6);
\draw [line width=1.5pt] (1,1)-- ++(2,0)-- ++(2,2)-- ++(0,2)-- ++(-2,0)-- ++(-2,-2)-- ++(0,-2);
\draw [line width=1pt,dotted] (2,2)-- ++(1,0)-- ++(1,1)-- ++(0,1)-- ++(-1,0)-- ++(-1,-1)-- ++(0,-1);
\draw [line width=1.5pt] (0,0)-- ++(1,1);
\draw [line width=1.5pt] (5,5)-- ++(1,1);
\draw [line width=1.5pt] (3,0)-- ++(0,1);
\draw [line width=1.5pt] (3,5)-- ++(0,1);
\draw [line width=1.5pt] (0,3)-- ++(1,0);
\draw [line width=1.5pt] (5,3)-- ++(1,0);

\draw [line width=1pt,dotted] (1,1)-- ++(1,1);
\draw [line width=1pt,dotted] (4,4)-- ++(1,1);
\draw [line width=1pt,dotted] (3,1)-- ++(0,1);
\draw [line width=1pt,dotted] (3,4)-- ++(0,1);
\draw [line width=1pt,dotted] (1,3)-- ++(1,0);
\draw [line width=1pt,dotted] (4,3)-- ++(1,0);

\draw (4,5) node {$\times$};
\draw (3.5,4) node {$\times$};
\end{tikzpicture}
&
\begin{tikzpicture}[line cap=round,line join=round,>=triangle 45,x=0.6cm,y=0.6cm]
\clip(0,0) rectangle (7,5);

\draw [line width=1.5pt] (0,1)-- ++(2,2)-- ++(3,0)-- ++(2,2);
\draw [line width=1.5pt] (2,3)-- ++(0,2);
\draw [line width=1.5pt] (5,0)-- ++(0,3);

\draw [line width=1pt,dotted] (0,0)-- ++(2,2)-- ++(3,0)-- ++(2,2);
\draw [line width=1pt,dotted] (2,2)-- ++(0,1);

\draw (3.5,2) node {$\times$};
\draw (3.5,3) node {$\times$};
\end{tikzpicture} \\
$(a)$ & $(b)$ \\
\end{tabular}
\caption{\label{figure point deformation cycle}Displacement of a point and the corresponding deformation of the cycle that it contains.}
\end{center}
\end{figure}

\medskip

When the point configuration becomes non-generic, we have at least a quadrivalent vertex that appears. We notice that when a point moves, the only part of the curve that moves is the cycle $\lambda_p$ containing it, and the edges adjacent to it vary their lengths. Let $p$ be the moving marked point. We distinguish two cases according to whether $\lambda_p=0$ or $\lambda_p\neq 0$.

\begin{itemize}[label=$\bullet$]
\item If $\lambda_p=0$, then the deformation only affects the cycle of the lifted curve $\tilde{h}:\widetilde{\Gamma}\to N_\RR$. The walls are dealt with using the invariance statement from \cite{itenberg2013block}. However, to be able to glue back the unbounded ends and get a curve inside $\TT A$, we have to assume tat the marked points sitting on the unbounded ends stay on the unbounded ends. In this case, the invariance exactly amounts to the local invariance statement from \cite{itenberg2013block}. One has the wall corresponding to a quadrivalent vertex, a marked point (not belonging to an unbounded end) meeting a vertex, or a cycle degenerating to a pair of parallel edges linking two quadrivalent vertices. These walls are depicted on Figure \ref{figure walls point deformation}.

\medskip

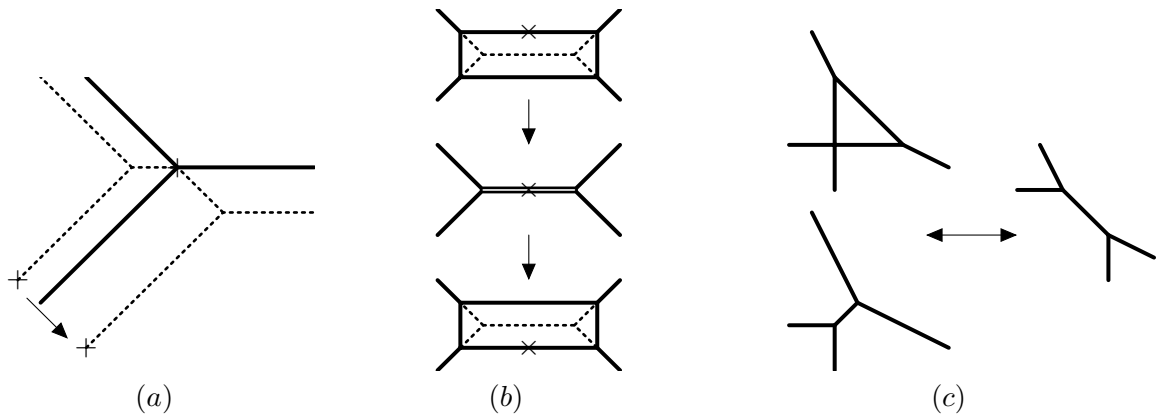
\begin{figure}[h]
\begin{center}
\begin{tabular}{ccc}
\begin{tikzpicture}[line cap=round,line join=round,>=triangle 45,x=0.6cm,y=0.6cm]
\clip(0,-0.5) rectangle (7,6);
\draw [line width=1.5pt] (1,1)-- ++(3,3)-- ++(3,0);
\draw [line width=1.5pt] (4,4)-- ++(-2,2);
\draw [line width=1pt,dotted] (0.5,1.5)-- ++(2.5,2.5)-- ++(4,0);
\draw [line width=1pt,dotted] (2,0)-- ++(3,3)-- ++(2,0);
\draw [line width=1pt,dotted] (5,3)-- ++(-1,1);
\draw [line width=1pt,dotted] (3,4)-- ++(-2,2);

\draw [->] (0.7,1.1)-- (1.6,0.2);

\draw (4,4) node {$+$};
\draw (0.5,1.5) node {$+$};
\draw (2,0) node {$+$};
\end{tikzpicture}
&
\begin{tikzpicture}[line cap=round,line join=round,>=triangle 45,x=0.6cm,y=0.6cm]
\clip(-2,0) rectangle (5,8);

\draw [line width=1.5pt] (0,6)-- ++(0.5,0.5)-- ++(0,1)-- ++(3,0)-- ++(0,-1)-- ++(-3,0);
\draw [line width=1.5pt] (0.5,7.5)-- ++(-0.5,0.5);
\draw [line width=1.5pt] (3.5,7.5)-- ++(0.5,0.5);
\draw [line width=1.5pt] (3.5,6.5)-- ++(0.5,-0.5);
\draw [line width=1pt,dotted] (0.5,6.5)-- ++(0.5,0.5)-- ++(2,0)-- ++(0.5,-0.5);
\draw [line width=1pt,dotted] (1,7)-- ++(-0.5,0.5);
\draw [line width=1pt,dotted] (3,7)-- ++(0.5,0.5);

\draw [line width=1.5pt] (0,3)-- ++(1,1)-- ++(-1,1);
\draw [line width=1.5pt] (4,3)-- ++(-1,1)-- ++(1,1);
\draw [line width=1pt,double] (1,4)-- ++(2,0);

\draw [line width=1.5pt] (0,0)-- ++(0.5,0.5)-- ++(0,1)-- ++(3,0)-- ++(0,-1)-- ++(-3,0);
\draw [line width=1.5pt] (0.5,1.5)-- ++(-0.5,0.5);
\draw [line width=1.5pt] (3.5,1.5)-- ++(0.5,0.5);
\draw [line width=1.5pt] (3.5,0.5)-- ++(0.5,-0.5);
\draw [line width=1pt,dotted] (0.5,0.5)-- ++(0.5,0.5)-- ++(2,0)-- ++(0.5,-0.5);
\draw [line width=1pt,dotted] (1,1)-- ++(-0.5,0.5);
\draw [line width=1pt,dotted] (3,1)-- ++(0.5,0.5);

\draw [->] (2,6)-- (2,5);
\draw [->] (2,3)-- (2,2);

\draw (2,0.5) node {$\times$};
\draw (2,4) node {$\times$};
\draw (2,7.5) node {$\times$};
\end{tikzpicture} 
&
\begin{tikzpicture}[line cap=round,line join=round,>=triangle 45,x=0.6cm,y=0.6cm]
\clip(-2,0) rectangle (9,8);

\draw [line width=1.5pt] (1,1)-- ++(0,-1);
\draw [line width=1.5pt] (1,1)-- ++(-1,0);
\draw [line width=1.5pt] (1,1)-- ++(0.5,0.5);
\draw [line width=1.5pt] (1.5,1.5)-- ++(-1,2);
\draw [line width=1.5pt] (1.5,1.5)-- ++(2,-1);

\draw [line width=1.5pt] (1,4)-- ++(0,2.5)-- ++(1.5,-1.5)-- ++(-2.5,0);
\draw [line width=1.5pt] (1,6.5)-- ++(-0.5,1);
\draw [line width=1.5pt] (2.5,5)-- ++(1,-0.5);

\draw [line width=1.5pt] (5,4)-- ++(1,0)-- ++(1,-1)-- ++(0,-1);
\draw [line width=1.5pt] (6,4)-- ++(-0.5,1);
\draw [line width=1.5pt] (7,3)-- ++(1,-0.5);

\draw [<->] (3,3)-- ++(2,0);
\end{tikzpicture} \\
$(a)$ & $(b)$ & $(c)$ \\
\end{tabular}
\caption{\label{figure walls point deformation}Displacement of a point and the corresponding deformation across the different walls. On $(a)$, a marked point that coincides with a vertex, on $(b)$, the flattening of a cycle, and on $(c)$, the appearance of a quadrivalent vertex.}
\end{center}
\end{figure}

Let us deal with the remaining case. If a marked point $q$ sitting on an unbounded end meets a vertex, there are three adjacent combinatorial types according to which edge adjacent to the vertex contains the marked point. This corresponds to a wall of type $(a)$ on Figure \ref{figure walls point deformation}. On one of of them, the complement $\Gamma\backslash h^{-1}(\mathcal{P}_s)$ is not simply connected anymore, and therefore cannot happen. In other words, the marked point $q$ has to stay on the cycle $\lambda_q$, and then moves from to the bounded edge belonging to $\lambda_p\cap\lambda_q$. The multiplicity is the same on both sides, and the two adjacent combinatorial types contribute solutions for $\P_{s\pm\varepsilon}$. To represent the deformation, we choose in the cutting set not $q$ but some point on the same edge as $q$, on the other side of the edges regarding the moving vertex.

\item Now, assume that $\lambda_p\neq 0$. It means that now we move a \textit{chord} (path between two unbounded ends) in the lifted curve. We proceed similarly. As long as the chord does not meet any other vertex sitting on an unbounded end, the invariance follows from the statements of \cite{itenberg2013block}. If the chord meets some marked point $q$ on an unbounded end, we do not take $q$ in the lifting set anymore but some other point on the edge, not on the side of the chord. The chords $\lambda_p$ and $\lambda_q$ overlap on a segment, and the marked point $q$ stays on the chord $\lambda_q$. The multiplicities also stay the same on both combinatorial type, and they provide solutions for $\P_{s\pm\varepsilon}$ respectively.
\end{itemize}
Finally, for each kind of wall, the gcd of the curves in the adjacent combinatorial types are the same. The result follows.
\end{proof}

\begin{rem}
Notice that we do not need the degenerations to occur one at a time. As the picture is purely local, they can happen at the same time. The only restriction is that a chord or a loop cannot approach a marked point from both sides at the same time: the edges containing it would contract to a point, and as the path $\P_s$ is chosen generically, it stays outside some fixed codimension $2$ sets, including the ones where the marked point sits at the place where the edge is contracted.
\end{rem}

\begin{rem}
Notice that going through the degenerations, we get that the gcd of the weights of the unbounded ends is preserved. This comes from the fact that in the moduli space of genus $g$ curves, the only way to change the gcd of the curves, is to pass through the codimension $1$ strata where a cycle is contracted to a point. And this is impossible by the genericity assumption on $\P_s$.

However, curves with different gcd can happen. For instance, consider the wall $(b)$ on Figure \ref{figure walls point deformation}, and all the edges of the curve be of weight divisible by $2$. Then, to produce a curve with a different gcd, deform the pair of parallel edges so that the two resulting edges have an odd weight. This is not considered in the proof because the slope of the edges stays the same along the deformation.
\end{rem}

\subsection{Independence of the choice of the surface}
\label{section invariance surface}

We now consider the dependence of the refined invariants when changing the lattice inclusion $j:\Lambda\hookrightarrow N_\RR$, \textit{i.e.} changing the matrix $S$ so that the relation $C S^T\in\S_2(\RR)$ for a fixed $C$ keeps being satisfied.

\begin{proof}[Proof of Theorem \ref{theorem torus invariant}]
Let $C\in \Lambda\otimes N$ be a class. Up to a change of basis of both $N$ and $\Lambda$, we can assume that $C$ is of the form $\begin{pmatrix}
d & 0 \\
0 & dn \\
\end{pmatrix}=d C_0$. Thus the condition on $S=\begin{pmatrix}
s_{11} & s_{12} \\
s_{21} & s_{22} \\ 
\end{pmatrix}$ is that $ns_{12}=s_{21}$. We choose a generic path $S_t=(s_{ij}(t))_{ij}$ between two generic choices of matrices $S_0$ and $S_1$ such that this relation is satisfied along the path. Here, a matrice $S$ is said to be generic if $\{ C'\in \Lambda\otimes N \text{ such that }C'S^T\in\S_2(\RR)\}=\ZZ C_0$. In other words, $C$ and its multiples are the only classes realizable. The path $S_t$ might have to cross the set of non-generic matrices to go from $S_0$ to $S_1$. Let $\Lambda_t=\mathrm{Im}S_t$ and $\TT A_t=N_\RR/\Lambda_t$.

\medskip

In addition to the path of matrices, for each $t$ we can choose a configuration of points $\P_t$ such that $\P_t$ varies continuously. Moreover, we can assume that for each $t$, the configuration $\P_t$ is generic as a point configuration inside $\TT A_t$ so that we know that the curves passing through $\P_t$ are simple. Notice that the choice $\P_t$ is generic inside $\TT A_t$ even when the matrix $S_t$ is not generic. Such a path of configurations exists because for every matrix $S$, the set of generic configurations is a dense open subset.

\medskip

Using the cutting process, curves inside $\TT A_t$ are obtained from curves inside $N_\RR$ satisfying some gluing condition. If we impose furthermore to pass through $\P_t$, the curves in $\TT A_t$ passing through $\P_t$ are in bijection with the curves inside $N_\RR$ passing through some fixed points (the lift of $\P_t$) and whose ends satisfy some moment conditions (for the gluing). When moving $t$, the solutions move. To prove the invariance, we only need to show that for each value $t_*$, on a neighborhood of $t_*$, the solutions to the enumerative problem deform and give refined counts that are equal.

\begin{rem}
The difference with usual proofs of tropical invariance, such as the proof of Theorem \ref{theorem point invariance}, is that here we avoid almost any kind of wall due to the genericity of the point configuration $\P_t$ inside $\TT A_t$. In fact, we are proving that the refined invariants $BG_{g,C,k}(\TT A)$ are locally constant on the moduli space of tropical torus with curves in the class $C$. As it is connected, it is therefore constant.
\end{rem}

The subtlety is that for a non-generic matrix $S_{t_*}$, there are several decompositions $C=C_1+C_2$, and the reducible curves on $\TT A_{t_*}$ can be deformed to irreducible curves in the class $C$ for nearby matrices $S_t$, $t\in]t_*-\varepsilon;t_*+\varepsilon[$. Let $h:\Gamma\to\TT A_{t_*}$ be a genus $g$ parametrized tropical curve in the class $C$ passing through the point configuration $\P_{t_*}$. As $\P_{t_*}$ is generic, and we do not assume irreducibility but only the genus, the irreducible components of $\Gamma$ are simple, and their genera add up to $g$.

\begin{itemize}[label=$\bullet$]
\item If $\Gamma$ is irreducible. Use the cutting process to lift it to a curve inside $N_\RR$ satisfying point constraints and moments constraints so that the gluing is possible. Then, as the curve is regular, a small deformation of the constraints yields a small deformation of the curve, and the multiplicity stays the same. In particular, a deformation of the gluing condition yields a small deformation of the curve and it provides solutions in the same combinatorial type for all $t\in]t_*-\varepsilon;t_*+\varepsilon[$.

\item If $\Gamma$ is reducible but with irreducible components that realize some class proportional to $C$, we reduce to the previous case by deforming each the corresponding components. Thus, it is not possible to deform them to an irreducible simple tropical curves: assume that the curve has two components. A deformation into an irreducible curve would result from the smoothing of a unique intersection point between the irreducible components: at least one to get a connected curve, and at most one to keep the right genus. Then, either the edge resulting from the smoothing is disconnecting and thus has slope $0$, or it is not disconnecting but has length $0$ due to the moment condition. See next point for more details.

\item Assume that $\Gamma=\Gamma_1\sqcup\cdots\sqcup\Gamma_r$ is reducible with $r$ connected components, not all of them realizing a class proportional to $C$. In particular, the matrix $S_{t_*}$ is not generic. We assume that none of the classes is proportional to $C$, since the previous points allow one to deform separately these components. Then, choose a minimal subset $\Q$ among the intersection points between the components $\Gamma_i$, and such that the graph obtained from $\Gamma$ adding edges corresponding to the intersection points has genus $g$, and the new components realize classes proportional to $C$. We obtain new parametrized tropical curves that have some contracted edges. The edges corresponding to the intersection points in $\Q$ are contracted since they are disconnecting. When deforming the torus and the curve along with it, each node gets resolved in one of the two possible ways, with an edge that is not contracted anymore.

\begin{expl}\label{example smoothing reducible 1}
Consider the class $C=\begin{pmatrix}
2 & 0 \\
0 & 3 \\
\end{pmatrix}$ inside the torus for $S=\begin{pmatrix}
8 & 0 \\
0 & 6 \\
\end{pmatrix}$. The corresponding tropical torus contains curves in the classes $\begin{pmatrix}
1 & 0 \\
0 & 0 \\
\end{pmatrix}$ and $\begin{pmatrix}
0 & 0 \\
0 & 1 \\
\end{pmatrix}$, which are not realizable for a small deformation of the torus. Consider the curve depicted on Figure \ref{figure smoothing reducible curve} $(b)$, which is comprised of $6$ irreducible components, all of genus $1$. The set $\Q$ is depicted with a $\times$. Then, deforming the matrix $S$, it is possible to get a curve of genus $6$ by smoothing the points of $\Q$, and get the curve on Figure \ref{figure smoothing reducible curve} $(c)$.
\end{expl}

\begin{figure}
\begin{center}
\begin{tabular}{ccc}
\begin{tikzpicture}[line cap=round,line join=round,>=triangle 45,x=0.8cm,y=0.8cm]
\tikzstyle{vertex}=[circle,draw]
\node[vertex] (E) at (0,2) {$\Gamma_5$};
\node[vertex] (A) at (1,0) {$\Gamma_1$};
\node[vertex] (C) at (2,2) {$\Gamma_3$};
\node[vertex] (B) at (3,0) {$\Gamma_2$};
\node[vertex] (D) at (4,2) {$\Gamma_4$};

\draw (A)--(C);
\draw (B)--(D);
\draw (C)--(B);
\draw (A)--(E);
\end{tikzpicture}
&
\begin{tikzpicture}[line cap=round,line join=round,>=triangle 45,x=0.5cm,y=0.5cm]
\clip(-1.2,-1) rectangle (8.5,7.5);

\draw [line width=1pt] (0,0)-- ++(8,0)-- ++(0,6)-- ++(-8,0)-- ++(0,-6);

\draw [line width=1.5pt] (2,0)-- ++(0,6);
\draw [line width=1.5pt] (4,0)-- ++(0,6);
\draw [line width=1.5pt] (6,0)-- ++(0,6);

\draw [line width=1.5pt] (0,2)-- ++(8,0);
\draw [line width=1.5pt] (0,4)-- ++(8,0);

\draw (2,2) node {$\times$};
\draw (2,4) node {$\times$};
\draw (4,4) node {$\times$};
\draw (6,2) node {$\times$};

\draw (0,2) node[left] {$\Gamma_1$};
\draw (0,4) node[left] {$\Gamma_2$};
\draw (2,6) node[above] {$\Gamma_3$};
\draw (4,6) node[above] {$\Gamma_4$};
\draw (6,6) node[above] {$\Gamma_5$};
\end{tikzpicture}
&
\begin{tikzpicture}[line cap=round,line join=round,>=triangle 45,x=0.5cm,y=0.5cm]
\clip(-0.5,-0.5) rectangle (9.5,7.5);

\draw [line width=1pt] (0,0)-- ++(8,0.75)-- ++(0.5,6)-- ++(-8,-0.75)-- ++(-0.5,-6);

\draw [line width=1.5pt] (2,0.1625)-- (2,2) ++(0.25,0.25)-- ++(0,1.75) ++(0.25,0.25)-- (2.5,6.1625);
\draw [line width=1.5pt] (4,0.375)-- (4,4.25) ++(0.5,0.5)-- (4.5,6.325);
\draw [line width=1.5pt] (6,0.4925)-- (6,2.25) ++(0.5,0.5)-- (6.5,6.4875);

\draw [line width=1.5pt] (0.166666,2)-- (2,2)-- ++(0.25,0.25)-- ++(3.75,0)-- ++(0.5,0.5)-- ++(1.666666,0);
\draw [line width=1.5pt] (0.3333333,4)-- (2.25,4)-- ++(0.25,0.25)-- ++(1.5,0)-- ++(0.5,0.5)-- (8.3333333,4.75);

\end{tikzpicture} \\
$(a)$ & $(b)$ & $(c)$ \\
\end{tabular}
\caption{\label{figure smoothing reducible curve} On $(b)$ a reducible curve with components in classes that do not survive the torus deformation, and on $(c)$ the corresponding deformation for a specific set of intersection points between the irreducible components, giving the graph from $(a)$.}
\end{center}
\end{figure}
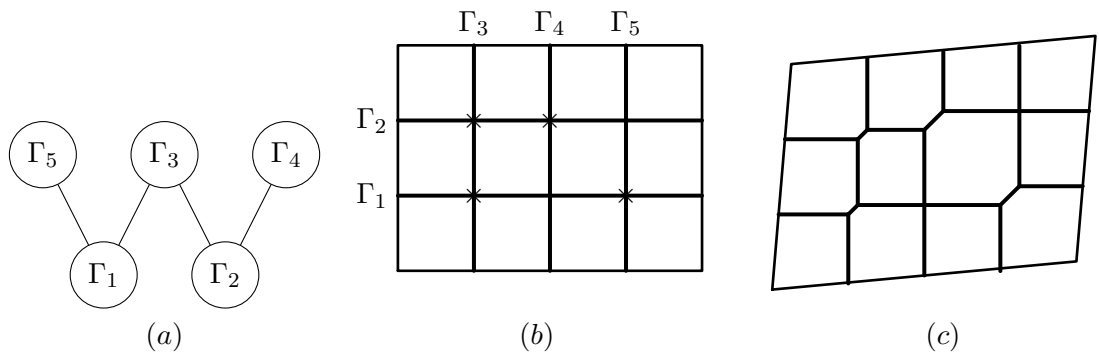

Consider the graph $\T$ consisting of the components $\Gamma_j$, where vertices are linked by an edge if the corresponding components share a point $q\in\Q$. This graph is a tree, otherwise the new graph would not have the right genus. 

\begin{expl}
The graph $\T$ corresponding to the example \ref{example smoothing reducible 1} is drawn on Figure \ref{figure smoothing reducible curve}. On  Figure \ref{figure smoothing reducible curve 2} is drawn a second example with the same base curve but a different set of node $\Q$.
\end{expl}

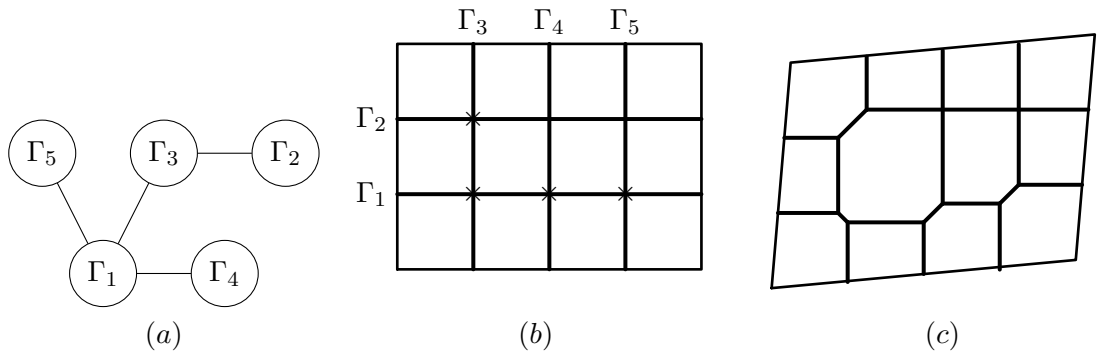
\begin{figure}[h]
\begin{center}
\begin{tabular}{ccc}
\begin{tikzpicture}[line cap=round,line join=round,>=triangle 45,x=0.8cm,y=0.8cm]
\tikzstyle{vertex}=[circle,draw]
\node[vertex] (E) at (0,2) {$\Gamma_5$};
\node[vertex] (A) at (1,0) {$\Gamma_1$};
\node[vertex] (C) at (2,2) {$\Gamma_3$};
\node[vertex] (B) at (3,0) {$\Gamma_4$};
\node[vertex] (D) at (4,2) {$\Gamma_2$};

\draw (A)--(C);
\draw (C)--(D);
\draw (A)--(B);
\draw (A)--(E);
\end{tikzpicture}
&
\begin{tikzpicture}[line cap=round,line join=round,>=triangle 45,x=0.5cm,y=0.5cm]
\clip(-1.2,-1) rectangle (8.5,7.5);

\draw [line width=1pt] (0,0)-- ++(8,0)-- ++(0,6)-- ++(-8,0)-- ++(0,-6);

\draw [line width=1.5pt] (2,0)-- ++(0,6);
\draw [line width=1.5pt] (4,0)-- ++(0,6);
\draw [line width=1.5pt] (6,0)-- ++(0,6);

\draw [line width=1.5pt] (0,2)-- ++(8,0);
\draw [line width=1.5pt] (0,4)-- ++(8,0);

\draw (2,2) node {$\times$};
\draw (2,4) node {$\times$};
\draw (4,2) node {$\times$};
\draw (6,2) node {$\times$};

\draw (0,2) node[left] {$\Gamma_1$};
\draw (0,4) node[left] {$\Gamma_2$};
\draw (2,6) node[above] {$\Gamma_3$};
\draw (4,6) node[above] {$\Gamma_4$};
\draw (6,6) node[above] {$\Gamma_5$};
\end{tikzpicture}
&
\begin{tikzpicture}[line cap=round,line join=round,>=triangle 45,x=0.5cm,y=0.5cm]
\clip(-0.5,-0.5) rectangle (9.5,7.5);

\draw [line width=1pt] (0,0)-- ++(8,0.75)-- ++(0.5,6)-- ++(-8,-0.75)-- ++(-0.5,-6);

\draw [line width=1.5pt] (2,0.1625)-- (2,1.75) ++(-0.25,0.25)-- (1.75,4) ++(0.75,0.75)-- (2.5,6.1625);
\draw [line width=1.5pt] (4,0.375)-- (4,1.75) ++(0.5,0.5)-- (4.5,6.325);
\draw [line width=1.5pt] (6,0.4925)-- (6,2.25) ++(0.5,0.5)-- (6.5,6.4875);

\draw [line width=1.5pt] (0.166666,2)-- (1.75,2)-- ++(0.25,-0.25)-- (4,1.75)-- ++(0.5,0.5)-- (6,2.25)-- ++(0.5,0.5)-- ++(1.666666,0);
\draw [line width=1.5pt] (0.3333333,4)-- (1.75,4)-- ++(0.75,0.75)-- (8.3333333,4.75);

\end{tikzpicture} \\
$(a)$ & $(b)$ & $(c)$ \\
\end{tabular}
\caption{\label{figure smoothing reducible curve 2} On $(b)$ a reducible curve with components in classes that do not survive the torus deformation, and on $(c)$ the corresponding deformation for a specific set of intersection points between the irreducible components, giving the graph from $(a)$.}
\end{center}
\end{figure}

Let $\Gamma_i$ be a leaf of $\T$ and apply the cutting process to it, for a cutting set that does not contain any point on the edge containing the intersection point $q$. The lifted curve inside $N_\RR$ is subject to some point constraints, and the gluing conditions from $\TT A_{t_*}$. Those are no longer satisfied when moving $t$ because the condition $C_i S_t^T$ is no longer satisfied. If $C_i S_t^T$ was symmetric, the gluing would be satisfied and the node would be deformable. Disconnect the edge containing $q$ and replace it by two unbounded ends. The new curve has genus one less, and now has the right number of constraints if we do not impose the two new unbounded ends to meet. The deformation forces the two ends to separate, and thus the appearance of a new bounded edge, and the combinatorial type is prescribed by the deformation of $\TT A_t$.

Proceeding inductively, we lift the curves corresponding to the vertices of $\T$ and prune the tree. When removing the vertex corresponding to a component, we cut open the curve at the intersection point and replace the bounded edge containing the intersection point by two unbounded ends, whose difference of moments is fixed by the leaf.

\begin{expl}
In the example of Figure \ref{figure smoothing reducible curve 2}, it means that when deformaing the torus, we first deform the components $\Gamma_2$, $\Gamma_4$ and $\Gamma_5$. Knowing the deformation of $\Gamma_2$, we deduce the deformation of $\Gamma_3$, and finally the deformation of $\Gamma_1$.
\end{expl}

For each choice of set $\Q$, among the $2^{|\Q|}$ adjacent combinatorial types resulting from the smoothing of the nodes of $\Q$, we only have two adjacent combinatorial types that are obtained by deformation of the torus. Clearly, the two deformations yield solutions for $\TT A_{t_\ast-\varepsilon}$ and $\TT A_{t_\ast-\varepsilon}$ respectively, and they have equal multiplicities since each intersection point becomes a pair of trivalent vertices with equal multiplicities on either side of the wall.
\end{itemize}
As the reducibility can be read on the graph, the result follows: if the tree $\T$ is not connected, it only yields reducible curves in $\TT A_t$ for nearby $t$. Otherwise, we get irreducible solutions.
\end{proof}

\begin{rem}
If  the tropical torus $\TT A$ is non-generic, there is still a relation between the invariants $BG_{g,C}(\TT A)$ and $BG_{g,C}$, which is the invariant for a generic choice of $\TT A$. The difference between both invariants comes from the appearance of reducible curves with classes not proportional to $C$: the value of $BG_{g,C}(\TT A)$ depends on the lattice $\{ C':C'S^T\in\S_2(\RR)\}$.
\end{rem}

\begin{rem}
The refined multiplicity of the tropical curves resulting from the smoothing of a reducible curve depends on the intersection index between the curves. With the refined multiplicity, having two intersection points of index $1$ is not the same as having a point of intersection index $2$, since the first yield a contribution $1$ to the refined multiplicity, while the second yield a $q+2+q^{-1}$. Surprisingly, the invariance statement tells us that the total count over all smoothings of all the reducible curve is invariant.
\end{rem}

\section{Examples}

Computation of the invariants requires to solve the tropical enumerative problem. Thus, we postpone the sophisticated examples to the third paper which provides an algorithm for the computation of the invariants. For sake of completeness, we include two examples to illustrate.

\begin{figure}[h]
\begin{center}
\begin{tikzpicture}[line cap=round,line join=round,>=triangle 45,x=0.5cm,y=0.5cm]
\clip(-0.5,-0.5) rectangle (21,5);

\draw [line width=1pt] (0,0)-- ++(19.8,0) --++ (0.2,1.2)-- ++(0,2.8)-- ++(-19.8,0) --++ (-0.2,-1.2)-- ++(0,-2.8);

\draw [line width=1.5pt] (0,1.5)-- (4,1.5)-- ++(0.2,0.2)-- (9,1.7)-- ++(0.6,0.6)-- (16,2.3)-- ++(0.2,0.4)-- (20,2.7);

\draw [line width=1.5pt] (4,0)-- (4,1.5) ++(0.2,0.2)-- (4.2,4);
\draw [line width=1.5pt] (9,0)-- (9,1.7) ++(0.6,0.6)-- (9.6,4);
\draw [line width=1.5pt] (9.2,0)-- ++(0,4);
\draw [line width=1.5pt] (9.4,0)-- ++(0,4);
\draw [line width=2pt] (16,0)-- (16,2.3) ++(0.2,0.4)-- (16.2,4);

\draw (16.2,3.5) node[right] {$2$};

\end{tikzpicture}

\caption{\label{figure stretched curve} A curve of genus $4$ in a long hexagon. It has one horizontal loop, and three vertical loops. Two of them make only one vertical round while the middle one makes three.}
\end{center}
\end{figure}
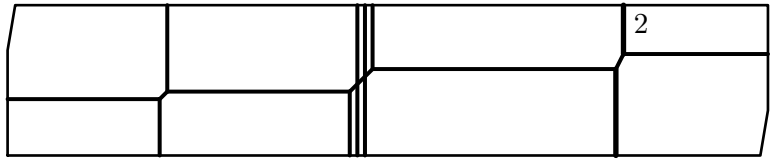

\begin{expl}
Assume that the class is $\begin{pmatrix}
1 & 0 \\
0 & n \\
\end{pmatrix}$ and that we are looking for curves of genus $g$ passing through $g$ points. Assuming that the abelian surface is represented by a fundamental domain that is a long hexagon and that the points constraints are horizontally spread, all the curves are of the following form: a loop that goes around the horizontal direction, and $g-1$ loops that go around the vertical direction. Each of the vertical loops is marked by a point, and one of the horizontal edges linking the vertical loops has a marked point. See Figure \ref{figure stretched curve} for an example of such a curve.

To choose such a curve, we need to choose the marked point that belongs to an horizontal edge and dispatch the degree $n$ among the $g-1$ vertical loops. For a choice of dispatching $a_1+\cdots +a_{g-1}=n$, we have several possibilities. Choosing $k_i|a_i$ such that the loop makes $k_i$ rounds around the vertical direction,  the $i$-th marked point has $k_i$ possible positions, and the loop has two vertices of multiplicity $\frac{a_i}{k_i}$. Thus, we get
\begin{align*}
N_{g,(1,n)} = & g \sum_{a_1+\cdots+a_{g-1}=n} \sum_{k_i|a_i}k_i\left(\frac{a_i}{k_i}\right)^2 \\
 = & g \sum_{a_1+\cdots+a_{g-1}=n} \prod_1^{g-1}a_i\sigma_1(a_i). \\
\end{align*}
We recover the result from \cite{bryan1999generating} giving the generating series of these numbers:
$$\sum_{n=1}^\infty N_{g,(1,n)}y^n = g DG_2(y)^{g-1},$$
where $G_2(y)$ is the Eisenstein series, and $D=y\frac{\dd}{\dd y}$. For the refined invariants, we get
\begin{align*}
BG_{g,(1,n)} = & g\sum_{a_1+\cdots+a_{g-1}=n}\sum_{k_i|a_i}k_i\left[\frac{a_i}{k_i}\right]^2. \\
\end{align*}
\end{expl}

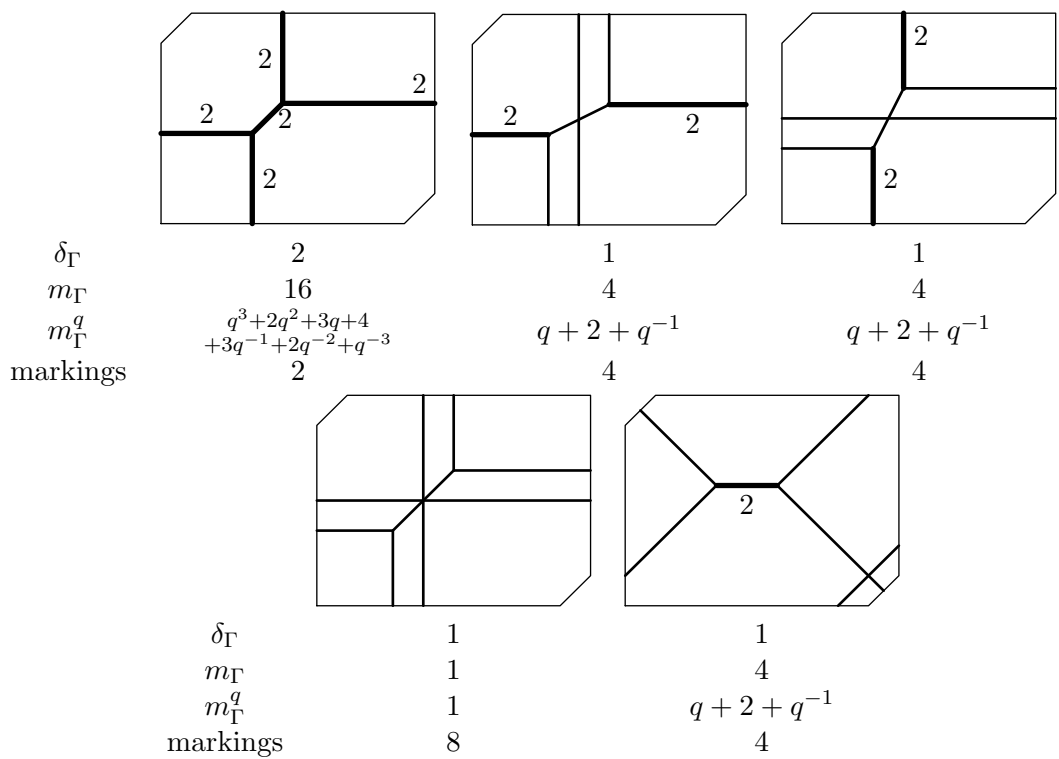
\begin{figure}[h]
\begin{center}
\begin{tabular}{cccc}
 & \begin{tikzpicture}[line cap=round,line join=round,>=triangle 45,x=0.4cm,y=0.4cm]

\draw [line width=0.5pt] (0,0)-- ++(8,0) --++ (1,1)-- ++(0,6)-- ++(-8,0) --++ (-1,-1)-- ++(0,-6);

\draw [line width=2pt] (0,3)-- (3,3)-- ++(1,1)-- (9,4);
\draw [line width=2pt] (3,0)-- (3,3);
\draw [line width=2pt] (4,7)-- (4,4);
\draw (3,1.5) node[right] {$2$};
\draw (1.5,3) node[above] {$2$};
\draw (8.5,4) node[above] {$2$};
\draw (4,5.5) node[left] {$2$};
\draw (3.5,3.5) node[below,right] {$2$};
\end{tikzpicture} &
\begin{tikzpicture}[line cap=round,line join=round,>=triangle 45,x=0.4cm,y=0.4cm]

\draw [line width=0.5pt] (0,0)-- ++(8,0) --++ (1,1)-- ++(0,6)-- ++(-8,0) --++ (-1,-1)-- ++(0,-6);

\draw [line width=2pt] (0,3)-- (2.5,3);
\draw [line width=1pt] (2.5,3)-- ++(2,1);
\draw [line width=2pt] (4.5,4)-- (9,4);
\draw [line width=1pt] (2.5,0)-- (2.5,3);
\draw [line width=1pt] (4.5,7)-- (4.5,4);
\draw [line width=1pt] (3.5,0)-- (3.5,7);
\draw (1.25,3) node[above] {$2$};
\draw (7.25,4) node[below] {$2$};
\end{tikzpicture} &
\begin{tikzpicture}[line cap=round,line join=round,>=triangle 45,x=0.4cm,y=0.4cm]

\draw [line width=0.5pt] (0,0)-- ++(8,0) --++ (1,1)-- ++(0,6)-- ++(-8,0) --++ (-1,-1)-- ++(0,-6);

\draw [line width=1pt] (0,2.5)-- (3,2.5)-- ++(1,2)-- (9,4.5);
\draw [line width=2pt] (3,0)-- (3,2.5);
\draw [line width=2pt] (4,7)-- (4,4.5);
\draw [line width=1pt] (0,3.5)-- (9,3.5);
\draw (3,1.5) node[right] {$2$};
\draw (4,6.25) node[right] {$2$};
\end{tikzpicture} \\
$\delta_\Gamma$ & $2$ & $1$ & $1$ \\
$m_\Gamma$ & $16$ & $4$ & $4$ \\
$m^q_\Gamma$ & $\substack{q^3+2q^2+3q+4\\+3q^{-1}+2q^{-2}+q^{-3}}$ & $q+2+q^{-1}$ & $q+2+q^{-1}$ \\
markings & $2$ & $4$ & $4$ \\
\end{tabular}
\begin{tabular}{ccc}
 & \begin{tikzpicture}[line cap=round,line join=round,>=triangle 45,x=0.4cm,y=0.4cm]

\draw [line width=0.5pt] (0,0)-- ++(8,0) --++ (1,1)-- ++(0,6)-- ++(-8,0) --++ (-1,-1)-- ++(0,-6);

\draw [line width=1pt] (0,2.5)-- (2.5,2.5)-- ++(2,2)-- (9,4.5);
\draw [line width=1pt] (2.5,0)-- (2.5,2.5);
\draw [line width=1pt] (4.5,7)-- (4.5,4.5);
\draw [line width=1pt] (0,3.5)-- (9,3.5);
\draw [line width=1pt] (3.5,0)-- (3.5,7);
\end{tikzpicture} &
\begin{tikzpicture}[line cap=round,line join=round,>=triangle 45,x=0.4cm,y=0.4cm]

\draw [line width=0.5pt] (0,0)-- ++(8,0) --++ (1,1)-- ++(0,6)-- ++(-8,0) --++ (-1,-1)-- ++(0,-6);

\draw [line width=1pt] (0,1)-- (3,4)-- (0.5,6.5);
\draw [line width=1pt] (7,0)--++ (2,2);
\draw [line width=1pt] (8.5,0.5)-- (5,4) --++(3,3);
\draw [line width=2pt] (3,4)-- (5,4);
\draw (4,4) node[below] {$2$};
\end{tikzpicture} \\
$\delta_\Gamma$ & $1$ & $1$ \\
$m_\Gamma$ & $1$ & $4$ \\
$m^q_\Gamma$ & $1$ & $q+2+q^{-1}$ \\
markings & $8$ & $4$ \\
\end{tabular}

\caption{\label{figure genus 2 deg 2 2}Genus $2$ curves and their multiplicities.}
\end{center}
\end{figure}

\begin{expl}
We compute the invariants for genus $2$ curves in the class $\begin{pmatrix}
2 & 0 \\
0 & 2 \\
\end{pmatrix}$. The combinatorial types of the curves are represented on Figure \ref{figure genus 2 deg 2 2}. For each of them, we compute the number of them that passes through $2$ fixed marked points. Thus, we recover
$$N_{2,(2,2)}=64+16+16+8+16=120,$$
and
$$BG_{2,(2,2)}=2q^3+4q^2+18q+40+18q^{-1}+4q^{-2}+2q^{-3}.$$
\end{expl}

\bibliographystyle{plain}
\bibliography{biblio}

\end{document}